\documentclass[10pt]{amsart}
\usepackage[pdfborder={0 0 0 [0 0 ]}]{hyperref}
\usepackage{amsmath,amsfonts,amssymb,amsthm}
\usepackage[capitalise]{cleveref}
\usepackage{graphicx}
\usepackage{color}
\usepackage[ps,all,arc,rotate,cmtip]{xy}
\usepackage{verbatim}
\usepackage[utf8]{inputenc}
\usepackage{tikz}

\long\def\symbolfootnote[#1]#2{\begingroup
\def\thefootnote{\fnsymbol{footnote}}\footnote[#1]{#2}\endgroup}
\numberwithin{equation}{section}

\theoremstyle{plain}
\newtheorem*{theorem*}{Theorem}
\newtheorem{thm}{Theorem}[section]

\newtheorem{lem}[thm]{Lemma}

\newtheorem{prop}[thm]{Proposition}

\newtheorem{cor}[thm]{Corollary}

\newtheorem{conj}[thm]{Conjecture}

\newtheorem{quest}[thm]{Question}

\theoremstyle{definition}
\newtheorem{defin}[thm]{Definition}
\newtheorem{dfn}[thm]{Definition}
\newtheorem{ex}[thm]{Example}

\theoremstyle{remark}

\newtheorem{rmk}[thm]{Remark}

\newtheoremstyle{maintheorem}{}{}{\itshape}{}{\bfseries}{}{.5em}{#1 \!\thmnote{#3}.}
\theoremstyle{maintheorem}
\newtheorem*{mainthm}{Theorem}
\newtheorem*{maincor}{Corollary}

\newcommand{\N}{\mathbb{N}}
\newcommand{\Q}{\mathbb{Q}}
\newcommand{\R}{\mathbb{R}}
\newcommand{\Z}{\mathbb{Z}}

\newcommand{\C}{\mathbb{C}}

\newcommand{\Aut}{\mathrm{Aut}}

\newcommand{\Nc}{\mathcal{N}}
\renewcommand{\S}{\mathcal{S}}
\newcommand{\U}{\mathcal{U}}
\renewcommand{\P}{\mathcal{P}}
\newcommand{\Wh}{\operatorname{Wh}}
\newcommand{\Hom}{\operatorname{Hom}}
\newcommand{\Map}{\operatorname{Map}}
\newcommand{\minsupp}{\operatorname{minsupp}}

\newcommand{\Cay}{\operatorname{Cay}}
\newcommand{\vol}{\operatorname{vol}}
\newcommand{\ab}{\operatorname{ab}}

\renewcommand{\bar}{\overline}
\renewcommand{\hat}{\widehat}
\newcommand{\norm}{\mathfrak{N}}
\newcommand{\supp}{\operatorname{supp}}
\newcommand{\im}{\operatorname{im}}

\renewcommand{\leq}{\leqslant}
\renewcommand{\geq}{\geqslant}

\def\into{\hookrightarrow}
\def\iff{if and only if }

\def\id{\mathrm{id}}
\def\GL{\mathrm{GL}}

\def\s-{\smallsetminus}

\def\D{\mathcal{D}}
\newcommand{\detc}{\operatorname{det^c_\D}}

\def\P{\mathbb{P}}
\def\Pol{\mathcal{P}}
\def\K{\mathbb{K}}

\newcommand{\fox}[2]{\frac{\partial #1}{\partial #2}}
\newcommand{\tolabel}[1]{\overset{#1}{\longrightarrow}}
\renewcommand{\det}{\operatorname{det}}
\newcommand{\tor}{\rho^{(2)}}
\renewcommand{\phi}{\varphi}
\newcommand{\Ddetc}{\operatorname{det_{\D(G)}^c}}
\newcommand{\Ddet}{\operatorname{det_{\D(G)}}}

\newcounter{dawidcomments}

\newcounter{flocomments}

\newcounter{hnncomments}

\title[Alexander and Thurston norms, and the BNS invariants for $F_n$-by-$\Z$]{Alexander and Thurston norms, and the Bieri--Neumann--Strebel invariants for free-by-cyclic groups}

\author{Florian Funke and Dawid Kielak}

\begin{document}

	\begin{abstract}
                \noindent We investigate Friedl--L\"uck's universal $L^2$-torsion for descending HNN extensions of finitely generated free groups, and so in particular for $F_n$-by-$\Z$ groups.
                This invariant induces a semi-norm on the first cohomology of the group which is an analogue of the Thurston norm for $3$-manifold groups.

                We prove that this Thurston semi-norm is an upper bound for the Alexander semi-norm defined by \mbox{McMullen}, as well as for the higher Alexander semi-norms defined by Harvey. The same inequalities are known to hold for $3$-manifold groups.

                We also prove that the Newton polytopes of the universal $L^2$-torsion of a descending HNN extension of $F_2$ locally determine the Bieri--Neumann--Strebel invariant of the group. We give an explicit means of computing the BNS invariant for such groups. As a corollary, we prove that the Bieri--Neumann-Strebel invariant of a descending HNN extension of $F_2$ has finitely many connected components.

                When the HNN extension is taken over $F_n$ along a polynomially growing automorphism with unipotent image in $\GL(n, \Z)$, we show that the Newton polytope of the universal $L^2$-torsion and the BNS invariant completely determine one another. We also show that in this case the Alexander norm, its higher incarnations, and the Thurston norm all coincide.
        \end{abstract}
	
\maketitle

\section{Introduction}
Whenever a free finite $G$-CW-complex $X$ is $L^2$-acyclic, i.e. its $L^2$-Betti numbers vanish, a secondary invariant called the $L^2$-torsion $\tor(X;\Nc(G))$ enters the stage \cite[Chapter 3]{Lueck2002}. It takes values in $\R$ and captures in many cases geometric data associated to $X$: If $X$ is a closed hyperbolic $3$-manifold, then it was shown by L\"uck and Schick \cite{LueckSchick1999} that
\[\tor(\widetilde{X};\Nc(\pi_1(X))) = -\frac{1}{6\pi}\cdot\vol(X)\]
and if $X$ is the classifying space of a free-by-cyclic group $F_n\rtimes_g\Z$, with $g\in\Aut(F_n)$, then $-\tor(\widetilde{X}; F_n\rtimes_g\Z)$ gives a lower bound on the growth rates of $g$, as shown by Clay \cite[Theorem 5.2]{Clay2015}.

Many generalisations of the $L^2$-torsion have been constructed, e.g. the $L^2$-Alexander torsion (by Dubois--Friedl--L\"uck \cite{DuboisEtal2014}) and $L^2$-torsion function, or more generally $L^2$-torsion twisted with finite-dimensional representations (by L\"uck \cite{Lueck2015}).

In a series of papers, Friedl and L\"uck~\cite{FriedlLueck2015a, FriedlLueck2015, FriedlLueck2015b} constructed the \emph{universal $L^2$-torsion} $\tor_u(X;\Nc(G))$ for any free finite $L^2$-acyclic $G$-CW-complex. It takes values in $\Wh^w(G)$, a weak version  of the Whitehead group of $G$ which is adapted to the setting of $L^2$-invariants. The Fuglede--Kadison determinant induces a map $\Wh^w(G)\to\R$ taking $\tor_u(X;\Nc(G))$ to $\tor(X;\Nc(G))$, and similar maps with $\Wh^w(G)$ as their domain take the universal $L^2$-torsion to the aforementioned generalisations of $L^2$-torsion.

Assuming that $G$ satisfies the Atiyah Conjecture, Friedl--Lück \cite{FriedlLueck2015b} construct a \emph{polytope homomorphism}
\[ \P\colon \Wh^w(G)\to \Pol_T(H_1(G)_f)\]
where $H_1(G)_f$ denotes the free part of the first integral homology of $G$, and $\Pol_T(H_1(G)_f)$ denotes the Grothendieck group of the commutative monoid whose elements are polytopes in $H_1(G)_f\otimes \R$ (up to translation) with pointwise addition (also called \emph{Minkowski sum}). The image of $-\tor_u(X;\Nc(G))$ under $\P$ is the \emph{$L^2$-torsion polytope of $X$}, denoted by $P_{L^2}(X;G)$. If $M\neq S^1\times D^2$ is a compact connected aspherical $3$-manifold with empty or toroidal boundary such that $\pi_1(M)$ satisfies the Atiyah Conjecture, then it is shown in ~\cite[Theorem 3.27]{ FriedlLueck2015b} that $P_{L^2}(\widetilde{M};\pi_1(M))$ induces another well-known invariant of $M$, the \emph{Thurston norm}
\[\| \cdot \|_T \colon H^1(M;\R) \to \R \]
This semi-norm was defined by Thurston~\cite{Thurston1986} and is intimately related to the question of the manifold fibering over the circle.

McMullen~\cite{McMullen2002} constructed an \emph{Alexander semi-norm} from the Alexander polynomial and showed that it provides a lower bound for the Thurston semi-norm. This was later generalised by Harvey \cite{Harvey05} to higher Alexander semi-norms \[\delta_n\colon H^1(M;\R) \to \R\]

Friedl--Lück's theory can also be applied to free-by-cyclic groups, or more generally to descending HNN extensions $G = F_n*_g$, with $g$ an injective endomorphism of $F_n$, and yields in this context a semi-norm
\[ \|\cdot\|_T\colon H^1(G;\R) \to \R\]
which we also call \emph{Thurston norm} due to the analogy with the $3$-manifold setting. We build a similar picture as for $3$-manifolds and prove that this semi-norm is an upper bound for McMullen--Harvey's Alexander semi-norms:

\begin{mainthm}[\ref{main thm alex vs thurston}]
	Let $G = F_n \ast_g$ be a descending HNN extension of $F_n$ with stable letter $t$, and let $\psi \in H^1(G;\R)$. Then
 \[
  \delta_1(\psi) \leqslant \delta_2(\psi) \leqslant \dots \leqslant \| \psi \|_T
 \]
If $\beta_1(G) \geqslant 2$, then also $\delta_0(\psi) \leqslant \delta_1(\psi)$.
If $\beta_1(G) = 1$, then $\delta_0(\psi) -| \psi(t) | \leqslant \delta_1(\psi)$.

When $\psi$ is fibred (that is $\ker \psi$ is finitely generated), then all the inequalities above become equalities.
\end{mainthm}

For a particular type of automorphism called \emph{UPG} (see \cref{def:upg}) we obtain an equality:

\begin{maincor}[\ref{main corollary 2}]
	Let $G = F_n\rtimes_g \Z$ with $n\geq 2$ and $g$ a UPG automorphism. Let $\phi\in H^1(G;\R)$. Then for all $k\geq 0$ we have
	\[ \delta_k(\phi) = \|\phi\|_T.\]
\end{maincor}

In the case of two-generator one-relator groups $G$ with $b_1(G) = 2$, the $L^2$-torsion polytope has been studied by Friedl--Tillmann~\cite{FriedlTillmann2015}. They established a close connection between $P_{L^2}(G) := P_{L^2}(EG; G)$ and the Bieri--Neumann--Strebel invariant $\Sigma(G)$. We prove similar results in our setting:

\begin{mainthm}[\ref{main thm BNS}]
        Let $g \colon F_2\to F_2$ be a monomorphism and let $G = F_2*_g$ be the associated descending HNN extension. Given $\phi \in\Hom(G,\R) \s- \{0\}$ such that $-\phi$ is not the epimorphism induced by $F_2 \ast_g$, there exists an open neighbourhood $U$ of $[\phi]$ in $S(G)$ and an element $d\in\D(G)^\times$ such that:
        \begin{enumerate}
                \item The image of $d$ under the quotient maps
                        \begin{equation*}
                                \D(G)^\times \to \D(G)^\times / [\D(G)^\times,\D(G)^\times] \cong K_1^w(\Z G) \to \Wh^w(G)
                        \end{equation*}
                        is $-\tor_u(G)$. In particular $P_{L^2}(G) = P(d)$ in $\Pol_T(H_1(G)_f)$.
                \item For every $\psi, \psi' \in\Hom(G,\R) \s- \{0\}$ which satisfy $[\psi], [\psi'] \in U$ and are $d$-equivalent, we have $[-\psi] \in \Sigma(G)$ if and only if $[-\psi'] \in \Sigma(G)$.
                 \end{enumerate}
\end{mainthm}

The $d$-equivalence is induced by the Newton polytopes associated to $d$ in a simple way (see \cref{polytope equivalent}).
As a corollary, we show (in \cref{finite bns}) that the BNS invariant for $G = F_2*_g$ as above has finitely many connected components.

Over arbitrary rank we can strengthen this result again for UPG automorphisms:

\begin{maincor}[\ref{main corollary}]
	Let $G = F_n\rtimes_g \Z$ with $n\geq 2$ and $g$ a UPG automorphism. Let $\phi\in H^1(G;\R)$. Then $[\phi]\in \Sigma(G)$ if and only if $F_\phi(P_{L^2}(G))=0$ in $\Pol_T(H_1(G)_f)$.
\end{maincor}

The face map $F_\phi$ is defined in \cref{face map}. This theorem is motivated by Cashen-Levitt's computation of the BNS invariant of such groups \cite[Theorem 1.1]{CashenLevitt2014}. 

\subsection*{Acknowledgements}
The first-named author was supported by GRK 1150 `Homotopy and Cohomology' funded by the DFG, the Max Planck Institute for Mathematics, and the Deutsche Telekom Stiftung.

The second-named author was supported by the ERC grant `Moduli' of Ursula Hamenst\"adt, and the SFB 701 `Spectral Structures and Topological Methods in Mathematics' of the Bielefeld University.

The authors would like to thank Stefan Friedl and Wolfgang L\"uck for helpful discussions, and the organisers of the `Manifolds and Groups' conference on Ventotene, where some of this work was conducted. The second-named author would also like to thank \L ukasz Grabowski.

\tableofcontents

\section{Preliminaries}\label{chapter: preliminaries}

\subsection{Descending HNN extensions}

\begin{dfn}
Let $G$ be a group, $H \leqslant G$ a subgroup, and $g \colon H \to H$ a monomorphism. The \emph{HNN extension associated to $g$} is the quotient of the free product of $G$ with $\langle t \rangle \cong \Z$ by
\[ \langle \! \langle \{ t^{-1}xt g(x)^{-1} \mid x \in H \} \rangle \! \rangle  \]
The element $t$ is called the \emph{stable letter} of the HNN extension. The HNN extension is called \emph{descending} if $H = G$. The natural epimorphism $G \ast_g \to \Z$, sending $t$ to 1 with $G$ in its kernel, is called the \emph{induced epimorphism}.
\end{dfn}

\begin{rmk}
Note that when $g \colon G \to G$ is an isomorphism, then $G \ast_g = G \rtimes_g \Z$ is a semi-direct product, or a $G$-by-$\Z$ group (since extensions with a free quotient always split).
\end{rmk}

In the final sections of this paper we will focus on descending HNN extensions $G = F_2\ast_g$. The following (well-known) result illustrates that this is somewhat less restrictive than it might seem.

\begin{prop}
Let $g \colon F_2 \to F_2$ be a monomorphism which is not onto. There exists $N \in \N$ such that for every $n \geqslant N$ there exists a monomorphism $g_n \colon F_n \to F_n$ such that
\[ F_2 \ast_g \cong F_n\ast_{g_n} \]
\end{prop}
\begin{proof}
We start by observing that Marshall Hall's theorem~\cite{Hall1949} tells us that there exists $N \in \N$ such that $g(F_2)$ is a free factor of a finite index subgroup $F_N$ of $F_2$. In fact it is easy to see (using the proof of Stallings~\cite{Stallings1983}) that this statement holds for any $n \geqslant N$ (here we are using the fact that $g$ is not onto; otherwise $N=2$ and we cannot take larger values of $n$).

Now $g$ factors as
\[
\xymatrix{ F_2 \ar[r]^a & F_n \ar[r]^b & F_2 }
\]
where $a$ embeds $F_2$ as a free factor, and $b$ is an embedding with image of finite index.
We let $g_n = a \circ b \colon F_n \to F_n$.

Next we construct the desired isomorphism. Let $t$ (resp. $s$) denote the stable letter of $F_2 \ast_g$ (resp. $F_n \ast_{g_n}$). Let $F_2 = \langle x_1, x_2 \rangle$ and $F_n = \langle x_1, \dots, x_n \rangle$; with this choice of generators the map $a$ becomes the identity.

Consider $h \colon F_2 \ast_g \to F_n\ast_{g_n}$ defined by
\[ h(x_i) = x_i \textrm{ and } h(t) = s\]
It is a homomorphism since
\[ t^{-1} x_i t = b(x_i)\]
and
\[ h(t^{-1}) h(x_i) h(t) = s^{-1} x_i s = b(x_i) = h(b(x_i))\]

Now consider $h' \colon F_n\ast_{g_n} \to  F_2 \ast_g$ induced by
\[ h'(x_i) = tb(x_i)t^{-1} \textrm{ and } h'(s) = t\]
It is clear that $h'$ is the inverse of $h$.
\end{proof}

\begin{rmk}
 Of course there is nothing special about $F_2$ in the above result. The proof works verbatim when $F_2$ is replaced by $F_m$ with $m \geqslant 2$.
\end{rmk}

\subsection{Dieudonn\'e determinant}\label{dieudonne}

While working with the universal $L^2$-torsion, the Dieudonn\'e determinant for matrices over skew-fields is of fundamental importance. We review here its definition and fix a so-called \emph{canonical representative}.

\begin{dfn}
Given a ring $R$, we will denote its group of units by $R^\times$.
\end{dfn}

\begin{dfn}[Dieudonn\'e determinant]
Given a skew field $\D$ and an integer $n$, let $M_n(\D)$ denote the ring of $n \times n$ matrices over $\D$. The \emph{Dieudonn\'e determinant} is a multiplicative map
\[ \det_\D \colon M_n( \D ) \to \D^\times / [\D^\times, \D^\times] \cup \{ 0 \} \]
defined as follows: First we construct its \emph{canonical representative}
\[\detc \colon M_n(\D) \to \D\]
and then set $\det_\D(A)$ to be image of $\detc(A)$ under the obvious map
\[ \D \to \D^\times / [\D^\times, \D^\times] \cup \{ 0 \}\]

The canonical representative is defined inductively:
\begin{itemize}
\item for $n=1$ we have $\detc((a_{11})) = a_{11}$;
\item if the last column of $A$ contains only zeros we set $\detc(A) = 0$;
\item for general $n$ (and a matrix $A$ with non-trivial last column) we first identify the bottommost non-trivial element in the last column of $A$. If this is $a_{nn}$ we take $P =\id$; otherwise, if the element is $a_{in}$, we take $P$ to be the permutation matrix which swaps the $i^{th}$ and $n^{th}$ rows of $A$; in either case we have $PA = A' = (a'_{ij})$ with $a'_{nn} \neq 0$.
Now we define $B = (b_{ij})$ by
\[ b_{ij} = \left\{ \begin{array}{ccc} 1 & \textrm{if} & i=j  \\
                     0 & \textrm{if} & i\neq j \textrm{ and } j< n  \\
                    -a'_{in}{a'_{nn}}^{-1}  & \textrm{if} & i \neq j=n
                    \end{array} \right. \]
This way we have
\[ BPA = A'' = (a''_{ij}) \]
with $a''_{in} = 0$ for all $i \neq n$. Let us set $C$ to be the $(n-1)\times(n-1)$ matrix $C = (a''_{ij})_{i,j < n}$. We define
\[ \detc(A) =  \det P  \cdot \detc (C) \cdot a''_{nn} \]
\end{itemize}
\end{dfn}
Note that the canonical representative $\detc$ is not multiplicative, but the determinant itself is, as shown by Dieudonn\'e~\cite{Dieudonne1943}.

It is immediate from the definition that when $\D$ is a commutative field, then the Dieudonn\'e determinant agrees with the usual determinant.

\begin{prop}[Formula for square matrices]
        \[ \det^c_\D \left( \begin{array}{cc} a & b \\ c& d \end{array} \right) = \left\{ \begin{array}{ccl} ad - bd^{-1}cd & \textrm{if} & d \neq 0\\ -bc  & \textrm{if} & d = 0 \end{array} \right. \]
\end{prop}


\subsection{Crossed products}

\begin{dfn}[Crossed product group ring]\label{def crossed product}
	Let $R$ be a ring and $G$ a group together with maps of sets $\phi \colon G \to \Aut(R)$ and $\mu \colon G\times G\to R^\times$ such that
	\begin{align*}
	\phi(g) \circ \phi(g') &= c(\mu(g,g'))\circ \phi(gg')\\
	\mu(g,g')\cdot \mu(gg',g'') &= \phi(g)(\mu(g', g''))\cdot \mu(g,g'g'')
	\end{align*}
	where $c\colon R^\times \to\Aut(R)$ maps an invertible element $r$ to the conjugation by $r$ on the left. Then the \emph{crossed product group ring} $R\ast G$ is the free left $R$-module with basis $G$ and multiplication induced by the rule
	\begin{equation}\label{twisted convolution}
	(\kappa g) \cdot (\lambda h) = \kappa \phi(g)(\lambda) \mu(g,h) gh
	\end{equation}
	for any $g,h \in G$ and $\kappa, \lambda \in R$.
	The conditions on $\mu$ and $\phi$ ensure the associativity of the multiplication, so that $R\ast G$ is indeed a ring.
\end{dfn}

Note that when $\phi$ and $\mu$ are trivial, we obtain the usual group ring $RG$.

\begin{ex}\label{crossed product of extension}
	Crossed product group rings appear naturally: Given an extension of groups
	\[ 1 \to K \to G \to Q \to 1\]
	we can identify $R G \cong (R K)\ast Q$, where the structure maps $\phi$ and $\mu$ are defined as follows: Let $s\colon Q\to G$ be a set-theoretic section of the given epimorphism $G\to Q$. Define
	\[\phi(q)\left(\sum_{k\in K}a_k\cdot k\right) = \sum_{k\in K} a_k\cdot s(q)ks(q)^{-1}\]
	and
	\[\mu(q,q') = s(q)s(q')s(qq')^{-1} \in K\]
	The isomorphism $(R K)\ast Q \to RG$ is given by
	\[
	\sum_{q\in Q}\lambda_q\cdot q\mapsto \sum_{q\in Q} \lambda_q\cdot s(q)
	\]
%

A case of particular interest occurs when $Q = \Z$. Under this assumption the section $s$ can be chosen to be a group homomorphism so that $\mu$ is trivial. The crossed product ring $(R K) \ast Q$ is then a ring of \emph{twisted Laurent polynomials} denoted $(R K)_t[z^{\pm}]$, where the twisting is determined by the automorphism $t = \phi(1)$. We will think of the variable $z$ as $s(1)$.
\end{ex}

\begin{dfn}
Given an element $x = \sum_{h\in G} \lambda_h \cdot h \in R \ast G$ we define its \emph{support} to be
\[ \supp(x) = \{h \in G \mid \lambda_h \neq 0 \} \]
Note that the support is a finite subset of $G$.
\end{dfn}

\subsection{Ore localisation}

We briefly review non-commutative localisation.

\begin{dfn}
 Let $R$ be a unital ring without zero-divisors, and let $T\subseteq R$ be a subset containing $1$ such that for every $s,t\in T$ we also have $st\in T$. Then $T$ satisfies the \emph{(left) Ore condition} if for every $r\in R$, $t\in T$ there are $r'\in R$, $t'\in T$ such that $t'r = r't$.

 One can then define a ring $T^{-1}R$, called the \emph{Ore localisation}, whose elements are  fractions $t^{-1}r$ with $r\in R, t\in T$, subject to the usual equivalence relation. There is an obvious ring monomorphism $R\to T^{-1}R$.
\end{dfn}

One instance of the Ore localisation will be of particular interest in this paper. If $G$ is an amenable group, $\D$ a skew field and $\D * G$ a crossed product which is a domain, then a result of Tamari \cite{Tamari1957} shows that $\D * G$ satisfies the left (and right) Ore condition with respect to the non-zero elements in $\D * G$. This applies in particular to the case where $G$ is finitely generated free-abelian. (Note that for untwisted group algebras $\K G$ without non-trivial zero divisors, the Ore condition for $\K G$ is equivalent to amenability of $G$ by a result of Bartholdi and the second-named author~\cite{Bartholdi2016}.)

Throughout the paper, we will only take the Ore localisation with respect to all non-zero elements of a ring.

\subsection{The Atiyah Conjecture and \texorpdfstring{$\D(G)$}{D(G)}}
In this section we review techniques which were originally developed for proving the Atiyah Conjecture, but have meanwhile been shown to be fruitful on many other occasions.

Given a group $G$, let $L^2(G)$ to denote the complex Hilbert space with Hilbert basis $G$ on which $G$ acts by translation. We use $\Nc(G)$ to denote the \emph{group von Neumann algebra of $G$}, i.e. the algebra of bounded $G$-equivariant operators on $L^2(G)$. Associated to any $\Nc(G)$-module $M$ (in the purely ring-theoretic sense), there is a \emph{von Neumann dimension} $\dim_{\Nc(G)}(M)\in[0,\infty]$ (see \cite[Chapter 6]{Lueck2002}).

\begin{conj}[Atiyah Conjecture]
	Let $G$ be a torsion-free group. Given a matrix $A\in \Q G^{m\times n}$, we denote by $r_A: \Nc(G)^m\to\Nc(G)^n$ the $\Nc(G)$-homomorphism given by right multiplication with $A$. Then $G$ satisfies the Atiyah Conjecture if for every such matrix the number $\dim_{\Nc(G)}(\ker(r_A))$ is an integer.
\end{conj}

The class of groups for which the Atiyah Conjecture is known to be true is large. It includes all free groups, is closed under taking directed unions, as well as extensions with elementary amenable quotients. Infinite fundamental groups of compact connected orientable irreducible $3$-manifolds with empty or toroidal boundary which are not closed graph manifolds are also known to satisfy the Atiyah Conjecture. For these statements and more information we refer to \cite[Chapter 3]{FriedlLueck2015}.

\begin{dfn}\label{division closure}
	Let $R\subseteq S$ be a ring extension. Then the \emph{division closure} of $R$ inside $S$ is the smallest subring $D$ of $S$ which contains $R$, such that every element in $D$ which is invertible in $S$ is already invertible in $D$. We denote it by $\D(R\subseteq S)$.
\end{dfn}

Let $\U(G)$ denote the algebra of affiliated operators of $\Nc(G)$. This algebra is carefully defined and examined in \cite[Chapter 8]{Lueck2002}. Note that $\Q G$ embeds into $\Nc(G)$, and therefore $\U(G)$, as right multiplication operators. Let $\D(G)$ denote the division closure of $\Q G$ inside $\U(G)$.

The following theorem appears in \cite[Lemma 10.39]{Lueck2002} for the case where $\Q G$ is replaced by $\C G$ in the above definitions, but the proof also carries over to rational coefficients.

\begin{thm}
\label{atiyah using D}
	A torsion-free group satisfies the Atiyah Conjecture if and only if $\D(G)$ is a skew field.
\end{thm}

It is known that if $H\subseteq G$ is a subgroup, then there is a canonical inclusion $\D(H)\subseteq\D(G)$.

\smallskip

Recall from \cref{crossed product of extension} that for an extension of groups
\[1\to K\to G\to Q\to 1\]
the group ring $\Z G$ is isomorphic to the crossed product $\Z K\ast Q$, where $Q$ acts on $\Z K$ by conjugation. When $G$ satisfies the Atiyah Conjecture, this action extends to an action on $\D(K)$ and one can identify the crossed product $\D(K)\ast Q$ with a subring of $\D(G)$ (see \cite[Lemma 10.58]{Lueck2002}). If $Q$ is finitely generated free-abelian, then $\D(K)\ast Q$ satisfies the Ore condition with respect to the non-zero elements $T$ and the Ore localisation admits by \cite[Lemma 10.69]{Lueck2002} an isomorphism
\begin{equation}\label{ore localisation iso}
T^{-1}\left( \D(K)\ast Q\right) \tolabel{\cong} \D(G)
\end{equation}

\subsection{Semifirs and specialisations}
\label{subsec: specialising}
In this section we review the notion of a specialisation, which allows us to compare skew-fields with given maps from a group algebra $\Q G$.

We start with the notion of a semifir. (In general, Cohn's book \cite{Cohn2006} contains a detailed discussion of many aspects of ring theory that will be of relevance to us.)

\begin{dfn}[Semifir]
A ring $R$ is a \emph{semifir} if every finitely generated right ideal of $R$ is free and of unique rank.
\end{dfn}

\begin{thm}[Dicks--Menal{~\cite{DicksMenal1979}}]
\label{semifir}
Let $R$ be a ring and $G$ a non-trivial group. Then $RG$ is a semifir \iff $R$ is a skew-field and $G$ is non-trivial and locally free.
\end{thm}

Now we introduce the notion of specialisation.

\begin{dfn}[Specialisation]
Let $R$ be a ring. An \emph{$R$-field} consists of a skew-field $\D$ and a ring morphism $\beta \colon R \to \D$. An $R$-field $\D$ is \emph{epic} if $\beta$ is an \emph{epimorphism}, that is,  if for any ring $S$ and any two ring morphisms $\sigma, \sigma' \colon \D \to S$, we have
\[
 \sigma \circ \beta = \sigma' \circ \beta \Longrightarrow \sigma = \sigma'
\]

Given two epic $R$-fields $\beta \colon R \to \D$ and $\beta' \colon R \to \D'$, a \emph{specialisation} of $\D$ to $\D'$ is a pair $(S,\sigma)$ where $S$ is a subring of $\D$ containing $\im \beta$, the map $\sigma \colon S \to D'$ is a ring map with $\sigma \circ \beta = \beta'$, and every element in $S$ not mapped to $0$ by $\sigma$ is invertible in $S$. The ring $S$ is called the \emph{domain} of the specialisation.
\end{dfn}

Note that what we call a specialisation is referred to as a `subhomomorphism' by Cohn; for Cohn a specialisation is an equivalence class of subhomomorphisms.

Note also that an epic $R$-field is in particular an $R$-module. Hence, given a matrix $M$ over $R$, we can talk about $M \otimes \D$; this is of course the same matrix as $\beta(M)$, where we apply the map $\beta$ to entries of $M$.

When $G$ is torsion-free and satisfies the Atiyah conjecture, then $\D(G)$ is an epic $\Q G$-field since it is the division closure of the image of $\Q G$ in $\U(G)$, see \cite[Corollary 7.2.2]{Cohn2006}.

\begin{thm}[Cohn~{\cite[Theorem 7.2.7]{Cohn2006}}]
 \label{spec criterion}
Let $R$ be a ring and let $\D, \D'$ be epic $R$-fields. The following are equivalent:
\begin{enumerate}
 \item There exists a specialisation from $\D$ to $\D'$.
 \item For every square matrix $M$ over $R$, if $M \otimes \D'$ is invertible over $\D'$ then $M \otimes \D$ is invertible over $\D$.
\end{enumerate}
\end{thm}

Cohn gives two further equivalent statements, but they will be of no importance to us.

 \smallskip
We now define a class of groups for which the skew-fields $\D(G)$ admit desirable specialisations.
\begin{dfn}[Specialising groups]
Let $\Phi$ be a collection of morphisms $\phi \colon G \to \R$. We say that $G$ is \emph{$\Phi$-specialising} if $G$ is torsion-free, satisfies the Atiyah Conjecture, and given any group epimorphism $\alpha \colon G \to \Gamma$ with $\Gamma$ torsion-free and elementary amenable such that every $\phi \in \Phi$ factors through $\alpha$, the $\Q G$-field $\D(G)$ admits a specialisation to the $\Q G$-field $\D(\Gamma)$, where the map $\Q G \to \D(\Gamma)$ is obtained by composing $\alpha \colon \Q G\to \Q \Gamma$ with the embedding $\Q \Gamma \to \D(\Gamma)$.

We say that a group $G$ is \emph{specialising} if $G$ is $\emptyset$-specialising.
\end{dfn}

Note that $\Phi$-specialising implies $\Psi$-specialising for $\Phi \subseteq \Psi$, and so specialising is the strongest property in this family of properties. On the other extreme, when $\Phi\supseteq H^1(G;\Z)$, then being $\Phi$-specialising means that we need to consider only those quotients $\Gamma$ which map onto $H_1(G;\Z)_f$, the free part of the abelianisation of $G$.

 The following is a combination of results of Cohn and Linnell.
 
 \begin{thm}
\label{loc free spec}
Locally free groups are specialising.
\end{thm}
\begin{proof}
Let $F$ denote a locally free group.
We start by observing that $\Q F$ is a semifir (by \cref{semifir} for non-trivial $F$, and by the fact that $\Q$ is a field for trivial $F$), and hence a \emph{Sylvester domain} by \cite[Proposition 5.1.1]{Cohn2006} (this last term is defined in \cite{Cohn2006}, but its precise meaning is not really important for us here).

Now let $M$ be an $n \times n$ matrix over $\Q F$. Suppose that there exist an $n \times m$  matrix $P$ and an $m \times n$ matrix  $Q$, both over $\Q F$, where $m<n$, and such that $M = PQ$. In such a situation $M$ is defined to be \emph{non-full}, and if no such $P$ and $Q$ exist, then $M$ is \emph{full}. Since $\Q F$ is a Sylvester domain, \cite[Theorem 7.5.12]{Cohn2006} gives us an \emph{honest} ring homomorphism $\beta\colon\Q F \to \D$, where $\D$ is an epic $\Q F$-field called \emph{the universal localisaton of $\Q F$ with respect to the set of full matrices}. `Honest' means precisely that if a square matrix
$M$ is full over $\Q F$, then $M \otimes \D$ is full over $\D$. Since $\D$ is a skew-field, it is easy to see that being full is the same as having non-zero determinant (and being invertible). Note also that $\beta$ is necessarily injective. 

Let $D'$ be any epic $\Q F$-field. Clearly, if $M$ is a square matrix over $\Q F$ with $M = PQ$, then $M \otimes \D' = P\otimes \D' \cdot Q\otimes \D'$. Thus,
if $M \otimes \D'$ is invertible, then $M$ itself is full, and therefore $M \otimes \D$ is full, and hence invertible. Thus, applying
\cref{spec criterion} tells us that $\beta\colon\Q F\to\D$ admits a specialisation to any epic $\Q F$-field (in Cohn's terminology, $\D$ is therefore \emph{the universal field of fractions}).

It remains to prove that $\D \cong \D(F)$. Since any group is the union of its finitely generated subgroups, there is an increasing sequence of finitely generated free subgroups $F_i$ of $F$ such that $F = \bigcup F_i$. By \cite[Lemma 10.83]{Lueck2002}, we have
\[
 \D(F) = \bigcup \D(F_i)
\]
Also, by \cite[Lemma 10.81]{Lueck2002}, $\D(F_i)$ is universally $\Sigma(\Q F_i\to\D(F_i))$-inverting (see \cite[Section 10.2.2]{Lueck2002} for the definition of this concept). Since $\Sigma(\Q F_i\to\D(F_i))\subseteq \Sigma(\Q F\to\D(F))$ is contained in the set of full matrices over $\Q F$, and $\beta\colon \Q F\to \D$ inverts all full matrices, there is a ring map $\gamma_i\colon\D(F_i)\to \D$ such that the square 
\[\xymatrix{
	\Q F_i \ar[r]\ar[d]  & \Q F \ar[d]^\beta\\
	\D(F_i)\ar[r]^{\gamma_i} & \D
}\]
commutes.

The map $\gamma_j$ agrees with $\gamma_i$ on $\D(F_i)$ for $j>i$; they thus fit together to give a map $\gamma\colon\D(F)\to D$ such that the triangle
\[\xymatrix{
	 & \Q F \ar[d]^\beta\ar[ld]\\
	 \D(F)\ar[r]^\gamma\ar[r]  & \D
}\]
commutes. But since $\beta$ is epic and $\gamma$ is necessarily injective, $\gamma$ must in fact be an isomorphism.

\end{proof}

\subsection{Universal \texorpdfstring{$L^2$}{L\texttwosuperior}-torsion}\label{sec: universal l2 torsion}
Let $G$ be a group satisfying the Atiyah Conjecture. In \cite[Definition 1.1]{FriedlLueck2015b}, Friedl and L\"uck define the \emph{weak $K_1$-group} $K_1^w(\Z G)$ as the abelian group generated by $\Z G$-endomorphisms $f\colon\Z G^n\to\Z G^n$ that become a weak isomorphism (a bounded injective operator with dense image) upon applying $-\otimes_{\Z G} L^2(G)$, subject to the usual relations in $K_1$. The above condition is equivalent to $f$ becoming invertible after applying $-\otimes_{\Z G} \D(G)$ (see \cite[Lemma 1.21]{FriedlLueck2015b}). The \emph{weak Whitehead group $\Wh^w(G)$ of $G$} is defined as the quotient of $K_1^w(\Z G)$ by $\{\pm g\mid g\in G\}$ considered as endomorphisms of $\Z G$ via right multiplication. An injective group homomorphism $i\colon G\to H$ induces maps
\begin{gather*}
	i_*\colon K_1^w(\Z G) \to K_1^w(\Z H)\\
	i_*\colon \Wh^w(G)\to \Wh^w(H)
\end{gather*}

\begin{ex}\label{ex: weak K1 of abelian}
	For $H$ a finitely generated free-abelian group, we have isomorphisms
	\[ K_1^w(\Z H) \cong K_1(T^{-1}(\Z H))\cong T^{-1}(\Z H)^\times\]
where $T$ denotes the set of non-trivial elements of $\Z H$. The first isomorphism is a special case of the main result of \cite{LinnellLueck2015} by Linnell--L\"uck, and the second one is well-known and induced by the Dieudonn\'e determinant over the field $T^{-1}(\Z H)$.
\end{ex}

A $\Z G$-chain complex is called \emph{based free} if every chain module is free and has a preferred basis. Given an $L^2$-acyclic finite based free $\Z G$-chain complex $C_*$, Friedl-L\"uck \cite[Definition 1.7]{FriedlLueck2015b} define the \emph{universal $L^2$-torsion of $C_*$}
\[\tor_u(C_*; \mathcal{N}(G)) \in K_1^w(\Z G)\]
in a similar fashion as the Whitehead torsion.

If $X$ is an $L^2$-acyclic finite free $G$-CW-complex, then its cellular chain complex $C_*(X)$ is finite and free, and we equip it with some choice of bases coming from the CW-structure. Since this is only well-defined up to multiplication by elements in $G$, the \emph{universal $L^2$-torsion $\tor_u(X;\Nc(G))\in\Wh^w(G)$ of $X$} is defined as the image of $\tor_u(C_*(X); \Nc(G))$ under the projection $K_1^w(\Z G)\to \Wh^w(G)$.

A finite connected CW-complex $X$ is \emph{$L^2$-acyclic} if its universal cover $\widetilde{X}$ is an $L^2$-acyclic $\pi_1(X)$-CW-complex. If this is the case, then the \emph{universal $L^2$-torsion of $X$} is
\[\tor_u(\widetilde{X}) := \tor_u(\widetilde{X};\mathcal{N}(\pi_1(X))) \in \Wh^w(\pi_1(X))\]

If $X$ is a (possible disconnected) finite CW-complex, then it is \emph{$L^2$-acyclic} if each path component is $L^2$-acyclic in the above sense. In this case, its \emph{universal $L^2$-torsion} is defined by
\[ \tor_u(\widetilde{X}) := (\tor_u(\widetilde C))_{C\in\pi_0(X)} \in \Wh^w(\Pi(X)) := \!\!\!\!\!\! \bigoplus_{C\in\pi_0(X)} \!\!\! \!\! \Wh^w(\pi_1(C))\]

A map $f\colon X\to Y$ of finite CW-complexes such that
\[\pi_1(f,x)\colon \pi_1(X,x)\to \pi_1(Y, f(x))\]
is injective for all $x\in X$ induces a homomorphism
\[ f_*\colon \Wh^w(\Pi(X)) \to \Wh^w(\Pi(Y))\]
by
\[ f_* := \big( (f|_C)_*\colon \Wh^w(\pi_1(C))\to \Wh^w(\pi_1(D))\big)_{C\in\pi_0(X)}\]
where $f(C)\subseteq D$.

\smallskip
The main properties of the universal $L^2$-torsion are collected in \cite[Theorem 2.5]{FriedlLueck2015b}, respectively \cite[Theorem 2.11]{FriedlLueck2015b}, of which we recall here the parts needed in this paper.

\begin{lem}\label{properties}
	\begin{enumerate}
		\item\label{homotopy invariance} Let $f\colon X\to Y$ be a $G$-homotopy equivalence of finite free $G$-CW-complexes. Suppose that $X$ or $Y$ is $L^2$-acyclic. Then both $X$ and $Y$ are $L^2$-acyclic and we get
		\[ \tor_u(X;\mathcal{N}(G)) - \tor_u(Y;\mathcal{N}(G)) = \zeta(\tau(f))\]
		where $\tau(f)\in\Wh(G)$ is the Whitehead torsion of $f$ and \[
\zeta\colon \Wh(G)\to\Wh^w(G)\] is the obvious homomorphism.
		\item\label{sum formula}	Let
		\[\xymatrix{
			X_0 \ar[r]\ar[d] \ar[dr]^{j_0} & X_1\ar[d]^{j_1}\\
			X_2\ar[r]^{j_2} & X
		}\]
		be a pushout of finite CW-complexes such that the top horizontal map is cellular, the left arrow is an inclusion of CW-complexes, and $X$ carries the CW-structure coming from the ones on $X_i$, $i=0,1,2$. Suppose that $X_i$ for $i= 0, 1,2$ is $L^2$-acyclic and that for any $x_i\in X_i$ the induced homomorphism $\pi_1(X_i,x_i)\to \pi_1(X,j_i(x_i))$ is injective. Then $X$ is $L^2$-acyclic and we have
		\[ \tor_u(\widetilde{X}) = (j_1)_*(\tor_u(\widetilde{X}_1)) + (j_2)_*(\tor_u(\widetilde{X}_2)) - (j_0)_*(\tor_u(\widetilde{X}_0)) \]
		\item \label{restriction} Let $p\colon X\to Y$ be a finite covering of finite connected CW-complexes. Let $p^*\colon \Wh^w(\pi_1(Y))\to\Wh^w(\pi_1(X))$ be the homomorphism induced by restriction with $\pi_1(p)\colon \pi_1(X)\to\pi_1(Y)$. Then $X$ is $L^2$-acyclic if and only if $Y$ is $L^2$-acyclic and in this case we have
		\[ \tor_u(\widetilde{X}) = p^*(\tor_u(\widetilde{Y}) )\]
	\end{enumerate}
\end{lem}

Next we apply this invariant to the groups we are interested in.

\begin{defin}\label{def: universal}
	Let $G$ be a group with a finite model for its classifying space $BG$, and let $g\colon G\to G$ be a monomorphism. Let $T$ be the mapping torus of the realisation $Bg\colon BG\to BG$. Given a factorisation $G *_g\tolabel{p}\Gamma\tolabel{q} \Z$ of the induced epimorphism, denote by $\bar{T}\to T$ the $\Gamma$-covering corresponding to $p$. Suppose that the classical Whitehead group $\Wh(\Gamma)$ of $\Gamma$ is trivial. Then $\bar{T}$ is $L^2$-acyclic \cite[Theorem 1.39]{Lueck2002}, and \cref{properties} (\ref{homotopy invariance}) implies that we get a well-defined invariant
	\[ \tor_u(G*_g,p) := \tor_u(\bar{T};\Nc(\Gamma)) \in \Wh^w(\Gamma) \]
	which only depends on $G, g$ and $p$, but not on the realisations. If $p= \id_G$, then we write $\tor_u(G*_g) = \tor_u(G*_g, \id_G)$.
\end{defin}
	A classical theorem of Waldhausen \cite[Theorem 19.4]{Waldhausen78} says that $\Wh(F_n*_g) = 0$, so that we may apply this in particular to the special case where $\Gamma = G\ast_g = F_n*_g$, and $p = \id$.

\subsection{The \texorpdfstring{$L^2$}{L\texttwosuperior}-torsion polytope}\label{sec: l2 torsion polytope}

Let $H$ be a finitely generated free-abelian group. An \emph{(integral) polytope} in $H\otimes_\Z \R$ is the convex hull of a non-empty finite set of points in $H$ (considered as a lattice inside $H\otimes_\Z \R$).

Given two polytopes $P_1$ and $P_2$ in $H\otimes_\Z \R$, their \emph{Minkowski sum} is defined as
\[ P_1 + P_2 := \{ x + y \in H\otimes_\Z \R\mid x\in P_1, y\in P_2\}\]
It is not hard to see that the Minkowski sum is \emph{cancellative} in the sense that $P_1 + Q = P_2 + Q$ implies $P_1 = P_2$. It turns the set of polytopes in $H\otimes_\Z \R$ into a commutative monoid with the one-point polytope $\{0\}$ as the identity. The \emph{(integral) polytope group of $H$}, denoted by  $\Pol(H)$, is defined as the Grothendieck completion of this monoid, so elements are formal differences of polytopes $P - Q$, subject to the relation
\[ P-Q = P'-Q'\Longleftrightarrow P+Q' = P' +Q\]
where on the right-hand side the symbol $+$ denotes the Minkowski sum. With motivation originating in low-dimensional topology, integral polytope groups have recently received increased attention, see \cite{ChaFriedlFunke2015, Funke2016}.

We define $\Pol_T(H)$ to be the cokernel of the homomorphism $H\to \Pol(H)$ which sends $h$ to the one-point polytope $\{h\}$. In other words, two polytopes become identified in $\Pol_T(H)$ \iff they are related by a translation with an element of $H$.

For a finite set $F\subseteq H$, we denote by $P(F)$ the convex hull of $F$ inside $H\otimes_\Z \R$.

\smallskip
Let $G$ be a torsion free group satisfying the Atiyah Conjecture. Then as before the integral group ring $\Z G$ embeds into the skew field $\D(G)$. Let $p\colon G\to H$ be an epimorphism onto a finitely generated free-abelian group $H$, and denote by $K$ the kernel of the projection $p$. Friedl-L\"uck \cite[Section 3.2]{FriedlLueck2015b} define a \emph{polytope homomorphism}
\begin{equation}\label{polytope homomorphism}
	\P\colon K_1^w(\Z G)\to \Pol(H)
\end{equation}
as the composition of the following maps: Firstly, apply the obvious map
\begin{equation}\label{poly1}
	K_1^w(\Z G)\to K_1(\D(G)), \;\; [f]\mapsto [\id_{\D(G)}\otimes_{\Z G} f]
\end{equation}
Since $\D(G)$ is a skew-field, the Dieudonn\'e determinant constructed in \cref{dieudonne} induces a map
\begin{equation}\label{poly2}
	\det_{D(G)}\colon K_1(\D(G)) \to \D(G)^\times/ [\D(G)^\times, \D(G)^\times]
\end{equation}
which is in fact an isomorphism (see Silvester~\cite[Corollary 4.3]{Silvester1981}). Finally, we use the isomorphism (\ref{ore localisation iso})
\begin{equation}\label{poly3}
	j\colon \D(G) \cong T^{-1}\left( \D(K)\ast H\right)
\end{equation}
For $x \in \D(K)\ast H$ we define $P(x) := P(\supp(x))\in \Pol(H)$. It is not hard to see that for two such elements $x_1,x_2$ we have $P(x_1x_2) = P(x_1) + P(x_2)$. We may therefore define a homomorphism
\begin{equation}\label{poly4}
	P\colon \big(T^{-1}\left( \D(K)\ast H\right)\big)^\times \to \Pol(H), \; t^{-1}s\mapsto P(s) - P(t)
\end{equation}

Since the target of $P$ is an abelian group, the composition $P\circ j|_{\D(G)^\times}$ factors through the abelianisation of $\D(G)^\times$. The polytope homomorphism announced in (\ref{polytope homomorphism}) is induced by the maps (\ref{poly1}), (\ref{poly2}), (\ref{poly3}) and (\ref{poly4}), and it does not depend on the choices used to construct the isomorphism (\ref{poly3}). We get an induced polytope homomorphism
\begin{equation}\label{poly_hom}
\P\colon \Wh^w(G) \to \Pol_T(H)
\end{equation}
If $x$ is an element in $\D(G)^\times$, we will henceforth use the isomorphism $j$ without mention and therefore denote the image of $x$ under $P\circ j|_{\D(G)^\times}$ simply by $P(x)$.

In the following definition we denote by $H_1(G)_f$ the free part of the abelianisation $H_1(G)$ of a group $G$.

\begin{defin}\label{def: l2 polytope}
	Let $X$ be a free finite $G$-CW-complex. We define the \emph{$L^2$-torsion polytope} $P_{L^2}(X;\Nc(G))$ of $X$ as the image of $-\tor_u(X;\Nc(G))$ under the polytope homomorphism (\ref{poly_hom}).
	
	Likewise, if $g\colon G\to G$ is a monomorphism of a group $G$ with a finite classifying space, and the obvious epimorphism $G*_g\to H_1(G*_g)_f$ factors through some $p\colon G*_g\to \Gamma$ such that $\Gamma$ satisfies the Atiyah Conjecture and $\Wh(\Gamma) = 0$, then the \emph{$L^2$-torsion polytope of $g$ relative to $p$}
	\[P_{L^2}(G*_g, p)\in \Pol_T(H_1(\Gamma)_f) = \Pol_T(H_1(G\ast_g)_f)\]
	is defined as the image of $-\tor_u(G*_g,p)$ under $\P\colon \Wh^w(\Gamma)\to \Pol_T(H_1(\Gamma)_f)$. If $p= \id_G$, then we just write $P_{L^2}(G*_g)$.
\end{defin}

We expect the $L^2$-torsion polytope to carry interesting information about the monomorphism $g$. Even for free groups we get an interesting invariant, which is new also for their automorphisms. On the other side of the universe of groups, the $L^2$-torsion polytope was shown to vanish if $X=EG$ is the finite classifying space of an amenable group $G$ that contains a non-abelian elementary amenable normal subgroup \cite{Funke2017}.

\subsection{The Alexander polytope}\label{Alex polytope}

The Alexander polynomial was first introduced by Alexander in~\cite{Alexander1928} as a knot invariant. Its definition was later extended by McMullen~\cite{McMullen2002} to all finitely generated groups in the following way.

Given a finite CW-complex $X$ with a basepoint $x$ and $\pi_1(X) = G$, consider the covering $\pi\colon\bar{X}\to X$ corresponding to the quotient map $p\colon G\to H_1(G)_f =: H$. The \emph{Alexander module} of $X$ is the $\Z H$-module
\[ A(X) = H_1(\bar{X},\bar{x},\Z)\]
where $\bar{x} = \pi^{-1}(x)$.

Now let $A$ be any finitely generated $\Z H$-module. Since $\Z H$ is Noetherian, we may pick a presentation
\[ \Z H^r \tolabel{M} \Z H^s \to A\to 0\]
The \emph{elementary ideal} $I(A)$  of $A$ is the ideal generated by all $(s-1)\times (s-1)$-minors of the matrix $M$. The \emph{Alexander ideal} of $X$ is $I(A(X))$, and the \emph{Alexander polynomial} $\Delta_X$ is defined as the greatest common divisor of the elements in $I(A(X))$. This invariant is well-defined up to multiplication by units in $\Z H$ and we will view it as an element in $\Wh^w(H)\cong T^{-1}(\Z H)/\{\pm h\mid h\in H\}$, where this isomorphism comes from Example \ref{ex: weak K1 of abelian}. Finally, the \emph{Alexander polytope} $P_A(X)$ is defined as the image of $\Delta_X$ under the polytope homomorphism
\[ \P\colon \Wh^w(H) \to \Pol_T(H)\]

The Alexander module and hence the Alexander polynomial depend only on the fundamental group, and we define $\Delta_G := \Delta_X$ and $P_A(G) := P_A(X)$ for any space with $\pi_1(X) = G$. This applies in particular to descending HNN extensions of finitely generated groups.

We emphasise that the Alexander polynomial is accessible from a finite presentation of $G$: We can take $X$ to be the presentation complex, so that the $\Z H$-chain complex of the pair $(\bar{X}, \bar{x})$ looks like
\[ 0 \to \Z H^r \tolabel{F} \Z H^s \to C_0(\bar{X})/C_0(\bar{x}) = 0\]
where $C_0$ denotes the group of zero chains and $F$ contains the Fox derivatives associated to the given presentation (see \cref{fox derivatives}). Thus $A(X)$ is the cokernel of the map $F$, which immediately gives a finite presentation of $A(X)$ as desired.

\subsection{Seminorms on the first cohomology}

Given a polytope $P\subseteq H\otimes_\Z\R$, we obtain a seminorm $\|\cdot\|_P$ on $\Hom(H, \R) \cong \Hom_\R(H\otimes_\Z\R, \R)$ by putting
\[\|\phi\|_P := \sup\{\phi(x)-\phi(y)\mid x,y\in P\}\]
It is clear that $\|\cdot\|_P$ remains unchanged when $P$ is translated within $H\otimes_\Z\R$. Moreover, if $Q$ is another such polytope, then we get for the Minkowski sum
\[ \|\phi\|_{P+Q} = \|\phi\|_P + \|\phi\|_Q\]
Thus we get a homomorphism of groups
\[ \norm\colon \Pol_T (H) \to \Map(\Hom(H,\R),\R),\; P-Q\mapsto \left( \phi\mapsto \|\phi\|_P - \|\phi\|_Q\right)\]
where $\Map(\Hom(H,\R), \R)$ denotes the group of continuous maps to $\R$ with the pointwise addition. In general, $\norm(P-Q)$ does not need to be a seminorm.

The following definition is due to McMullen \cite{McMullen2002}.

\begin{defin}
	If $G$ is a finitely generated group, then the \emph{Alexander norm}
	\[ \|\cdot\|_A\colon H^1(G;\R)\to\R\]
	is defined as the image of the Alexander polytope $P_A(G)$ under $\norm$.
\end{defin}

If $G$ is the fundamental group of a compact connected orientable $3$-manifold $M$, the first cohomology $H^1(M;\R) = H^1(G;\R)$ carries another well-known seminorm $\|\cdot\|_T$, called the \emph{Thurston seminorm}. It was first defined and examined by Thurston \cite{Thurston1986} and is closely related to the question of whether (and how) $M$ fibres over the circle. One of the main results of \cite[Theorem 3.27]{FriedlLueck2015b} is the following.

\begin{thm}\label{l2 torsion polytope and thurston}
	Let $M\neq S^1\times D^2$ be a compact connected aspherical $3$-manifold such that $\pi_1(M)$ satisfies the Atiyah Conjecture. Then the image of the $L^2$-torsion polytope $P_{L^2}(\widetilde{M};\pi_1(M))$ under $\norm$ is the Thurston seminorm $\|\cdot \|_T$.
\end{thm}

Motivated by this result, we make the following definition.

\begin{defin}\label{def:Thurston for fbc}
	Let $G = F_n*_g$ for a monomorphism $g\colon F_n\to F_n$. We call the image of the $L^2$-torsion polytope $P_{L^2}(G)\in \Pol_T(H_1(G)_f)$ as defined in \cref{def: l2 polytope} under $\norm$ the \emph{Thurston seminorm on $G$} and denote it by
	\[\|\cdot \|_T\colon H^1(G;\R)\to \R\]
\end{defin}

In order for this definition to make sense, we need to argue that HNN extensions of free groups satisfy the Atiyah Conjecture.

To this end, observe that $G$ fits into the extension
\[ 0\to  \langle \! \langle F_n \rangle \! \rangle \to G \to \Z \to 0\]
By the work of Linnell (see~\cite[Theorem 10.19]{Lueck2002}), we know that the Atiyah conjecture holds for $F_n$, is stable under taking directed unions, and so holds for $\langle \! \langle F_n \rangle \! \rangle$, and is stable under taking extensions with elementary amenable quotients, and thus holds for $G$.

The proof that the terminology \emph{seminorm} in the above definition is justified needs to be postponed to \cref{indeed seminorm}.

\smallskip

In \cite{Harvey05} Harvey generalised McMullen's work and defined \emph{higher Alexander norms}
\[\delta_k\colon H^1(G;\R)\to\R\]
for any finitely presented group $G$, where $\delta_0 = \|\cdot\|_A$. While we do not need the precise definition of $\delta_k$, the following ingredient will be needed throughout the paper.

\begin{defin}\label{rational derived}
	The \emph{rational derived series}
	\[ G = G_r^{0} \supseteq G_r^{1} \supseteq G_r^{2} \supseteq \cdots \]
	of a group $G$ is inductively defined with $G_r^{k+1}$ being the kernel of the projection
	\[ G_r^k \to H_1(G_r^k)_f\]
\end{defin}

Note that the quotients $\Gamma_k := G/ G_r^{k+1}$ are torsion free and solvable, and so
\[\Wh(\Gamma_k) = 0\]
since solvable groups satisfy the $K$-theoretic Farrell--Jones Conjecture by a result of Wegner \cite{Wegner2013}. Moreover, $\Gamma_k$ satisfies the Atiyah Conjecture by the work of Linnell (see~\cite[Theorem 10.19]{Lueck2002}). Thus, given $G = F_n*_g$, \cref{def: universal} and \cref{def: l2 polytope} produce an $L^2$-torsion polytope $P_{L^2}(G, p_k)$ for the projections
\[ p_k\colon G\to \Gamma_k \]

The next result is not explicitly stated in \cite{FriedlLueck2015, FriedlLueck2015b}, but we will indicate how it directly follows from it.

\begin{thm}\label{higher Alex polytopes}
	Let $G = F_n*_g$ be a descending HNN extension and let
	\[ p_k\colon G\to \Gamma_k := G/ G_r^{k+1}\]
	be the obvious projection. Then the image of the $L^2$-torsion polytope $P_{L^2}(G, p_k)$ under $\norm$ is the higher Alexander norm $\delta_k$, unless $b_1(G) = 1$ and $k=0$.
\end{thm}
\begin{proof}
	Let $\nu_k\colon \Gamma_k\to H_1(G)_f$ be the natural projection. There is an obvious analogue of \cite[Theorem 8.4]{FriedlLueck2015} for  HNN extensions of free groups which says that for $\phi\colon H_1(G)_f\to\Z$ we have an equality
	\[\delta_k(\phi) = -\chi^{(2)}(T; p_k, \phi\circ\nu_k)\]
	where $T$ denotes the mapping torus of a realisation of $g$. The right-hand side denotes the twisted $L^2$-Euler characteristic defined and examined in \cite{FriedlLueck2015}.
	
	On the other hand, a similar argument as in the proof \cref{l2 torsion polytope and thurston} (see the proof of \cite[Theorem 3.27]{FriedlLueck2015b}) shows that
	\[ \norm(P_{L^2}(G, p_k))(\phi) = \norm(\P(-\tor_u(G,p_k)))(\phi) = -\chi^{(2)} (T;p_k, \phi\circ\nu_k) \qedhere\]
\end{proof}

Motivated by this result, we introduce new terminology.

\begin{defin}
	Let $G = F_n*_g$ be a descending HNN extension and let
	\[ p_k\colon G\to \Gamma_k := G/ G_r^{k+1}\]
	be the obvious projection. Then we call  $P_{L^2}(G, p_k)$ the \emph{higher Alexander polytopes}.
\end{defin}

The Thurston and higher Alexander seminorms satisfy well-known inequalities for compact orientable $3$-manifolds by the work of McMullen and Harvey \cite{McMullen2002, Harvey05, Harvey06}. We use their characterisation in terms of polytopes to prove an analogue in the case of
descending HNN extensions of free groups.
This will be the main result of \cref{sec: inequalities}.

\subsection{The Bieri--Neumann--Strebel invariant \texorpdfstring{$\Sigma(G)$}{Sigma(G)}}\label{sec:bns}

We first recall one of the definitions of the BNS-invariant $\Sigma(G)$, see \cite[Chapter A2.1]{Strebel2012}.

\begin{dfn}[The BNS invariant]
	Let $G$ be a group with finite generating set $\S$. The positive reals $\R_{>0}$ act on $\Hom(G,\R)\setminus\{0\}$ by multiplication. The quotient will be denoted by
	\[ S(G) = \left(\Hom(G,\R)\setminus\{0\}\right)/\R_{>0}\]
	Given a class $[\phi]\in S(G)$, let $\Cay(G,S)_\phi$ denote the subgraph of the Cayley graph of $G$ with respect to $S$ that is induced by the vertex subset $\{g\in G\mid \phi(g)\geq 0\}$.
	The \emph{BNS invariant} or \emph{$\Sigma$-invariant} is the subset	
	\[\Sigma(G) = \{ [\phi]\in S(G)\mid \Cay(G,S)_\phi\text{ is connected} \}\]
\end{dfn}

Note that $S(G)$, with the quotient topology, is naturally homeomorphic to the unit sphere in $H^1(G;\R)$. The invariant $\Sigma(G)$ is an open subset thereof (see~\cite[Theorem A]{Bierietal1987}).

For rational points in $S(G)$ we have a more tangible characterisation.

\begin{thm}[{\cite[Proposition 4.3]{Bierietal1987}}]
	Let $\phi\colon G\to\Z$ be an epimorphism. Then $[-\phi]\in\Sigma(G)$ if and only if
	$G$ can be identified with a descending HNN-extension over a finitely generated subgroup, so that $\phi$ is the epimorphism induced by the HNN-extension.
\end{thm}

\begin{dfn}[Sikorav--Novikov completion]\label{def:sikorav-novikov}
	Let $G$ be a group and $\phi\in H^1(G;\R)$. Then the \emph{Sikorav--Novikov completion} $\widehat{\Z G}_\phi$ is defined as the set
	\[ \widehat{\Z G}_\phi := \left\{\sum_{g\in G} x_g\cdot g\mid \forall C\in\R:\; \left\{g \in G\mid \phi(g) < C\text{ and } x_g\neq 0\right\}\text{ is finite}\right\}\]
\end{dfn}


It is easy to verify that the usual convolution turns $\widehat{\Z G}_\phi$ into a ring which contains $\Z G$. The reason why we are interested in the Sikorav--Novikov completion is the following criterion to detect elements in the BNS-invariant.

\begin{thm}\label{bns criterion}
	Given a finitely generated group $G$, for a non-zero homomorphism $\phi \colon G \to \R$
	we have $[-\phi] \in \Sigma(G)$ \iff
	\[
	H_0(G; \widehat{\Z G}_\phi) =0 \textrm{ and } H_1(G; \widehat{\Z G}_\phi )=0
	\]
\end{thm}
\begin{proof}
	This is originally due to Sikorav \cite{Sikorav1987}, see also \cite[Theorem 4.3]{FriedlTillmann2015} for a sketch of the proof.
\end{proof}

\begin{rmk}
 In fact we are only discussing the \emph{first BNS invariant}
 \[\Sigma^1(G;\Z) = -\Sigma(G)\]
 It is easily deducible from the full result of Sikorav that for descending HNN extensions of free groups the higher BNS invariants $\Sigma^n(G;\Z)$ all coincide with $\Sigma^1(G;\Z)$.
\end{rmk}

\begin{dfn}\label{alpha}
	We define $\mu_\phi \colon \hat{\Z G}_\phi \to \Z G$ in the following way: Let
	\[x = \sum_{g\in G} x_g\cdot g\in \hat{\Z G}_\phi\]
	and let
	\[ S = \big\{g\in \supp(x)\mid \phi(g) = \min\{ \phi(\supp(x))\} \big\}\]
	Then we let
	\[ \mu_\phi(x) = \sum_{g\in S} x_g\cdot g\]
\end{dfn}

It is easy to see that $\mu_\phi$ respects the multiplication in $\hat{\Z G}_\phi$.

The following criterion to detect units in $\hat{\Z G}_\phi$ is well-known; we include a proof here for the sake of completeness. Note that the Sikorav-Novikov completion is a domain, so being left-invertible is equivalent to being right-invertible, and so is equivalent to being a unit.

\begin{dfn}
A group $G$ is called \emph{indicable} if it admits an epimorphism onto $\Z$. The group is \emph{locally indicable} if all of its finitely generated subgroups are indicable.
\end{dfn}

\begin{lem}\label{invertible}
	Let $G$ be a locally indicable group and $x \in \widehat{\Z G}_\phi$. Then $x$ is a unit in $\widehat{\Z G}_\phi$ \iff $\mu_\phi(x)$ is of the form $\pm h$ for some $h\in G$.
\end{lem}
\begin{proof}
	If $x$ has an inverse $y\in \widehat{\Z G}_\phi$, then
	\[ 1 = \mu_\phi(1) = \mu_\phi(x)\mu_\phi(y)\]
	The latter is an equation in $\Z G$, where the only units are of the form $\pm h$ since $G$ is locally indicable \cite[Theorem 13]{Higman1940}.
	
	Conversely, write $x = \sum_{g\in G} x_g\cdot g$ and write $G_k$ for the (finite) set of elements $g\in G$ with $g\in \supp(x)$ and $\phi(g) = k$. After multiplying with the unit $\mu_\phi(x)^{-1}$, we may assume without loss of generality that $G_k = \emptyset$ for $k<0$, $G_0 \neq\emptyset$, and $\mu_\phi(x) = 1$, so
	\[ x = 1 + \sum_{g\in G_1} x_g\cdot g + \sum_{g\in G_2} x_g\cdot g + \dots\]
	It is now easy to successively build a left-inverse beginning with
	\[ 1 - \sum_{g\in G_1} x_g\cdot g + \left(\sum_{g\in G_1} x_g\cdot g\right)^2 - \sum_{g\in G_2} x_g\cdot g +\dots \qedhere\]
\end{proof}

Finally we verify that the above characterisation of units in $\widehat{\Z G}_\phi$ is applicable for the groups of our interest.

\begin{lem}
\label{G is locally indicable}
Let $g \colon F_n \to F_n$ be a monomorphism. Then the associated descending HNN extension is locally indicable.
\end{lem}
\begin{proof}
Let $G = F_n \ast_g$ denote the descending HNN extension, and let $\psi$ be the induced epimorphism to $\Z$.

We start by noting that $G$ is locally indicable \iff the normal closure of $F_n$ inside $G$ is, since this normal closure is the kernel of $\psi$, and the image of $\psi$ is a free-abelian group, and thus locally indicable.
Now, since $G$ is a descending HNN extension, every finitely generated subgroup of $\ker \phi$ lies in a copy of $F_n$, which is locally indicable. Hence $G$ is locally indicable.
\end{proof}

\subsection{Fox calculus}
\label{fox derivatives}

In order to start computing, we introduce as a last tool Fox derivatives (defined by Fox in~\cite{Fox1953}).

\begin{dfn}
Let $F_n$ be a free group generated by $s_1, \dots, s_n$, and let $w$ be a word in the alphabet $\{s_1, \dots, s_n \}$. We define the \emph{Fox derivative} $\fox w {s_i} \in \Z G$ of $w$ with respect to $s_i$ inductively: we write $w = vt$ where $t$ is one of the generators or their inverses, and $v$ is strictly shorter than $w$, and set
\[ \fox w {s_i} = \left\{ \begin{array}{lcl}
                           \fox v {s_i} & & t \not \in \{s_i, s_i^{-1} \} \\
                           \fox v {s_i} + v & \textrm{ if }& t  = s_i  \\
                           \fox v {s_i}  - w & & t = s_i^{-1}  \\
                          \end{array} \right. \]
This definition readily extends first to elements $w \in F_n$, and then linearly to elements of $\Z F_n$, forming a map $\fox w {s_i}\colon \Z F_n \to \Z F_n$.
\end{dfn}

The following equation is known as the fundamental formula of Fox calculus \cite[Formula (2.3)]{Fox1953}.

\begin{prop}
\label{foxIntegration}
	Let $w \in F_n$ be any word, and let $s_1, \dots, s_n$ be a generating set of $F_n$. Then we have
	\[ \sum_{i=1}^{n} \fox{w}{s_i}\cdot (1-s_i) = 1-w \]
\end{prop}

\section{The invariants for descending HNN extensions of free groups}\label{sec: free-by-cyclic}
\label{consequences}

In this section we describe the Alexander polynomial and the universal $L^{2}$-torsion in more explicit terms for descending HNN extensions of finitely generated free groups. The computations in this chapter follow from the general properties of the invariants, but we thought it worthwhile to collect them here in order to emphasise that a close connection between the invariants should not come as a complete surprise.

Let us first observe the following.

\begin{lem}
 Let $G$ be a descending HNN extension $G = F_n\ast_g$. Pick a finite classifying space $B F_n$ for $F_n$, and a realisation $Bg\colon BF_n\to BF_n$. Then the mapping torus $T_{Bg}$ of $Bg$ is a classifying space for $G$.
\end{lem}
\begin{proof}
It is well-known that $\pi_1(T_{Bg}) = G$. For the higher homotopy groups we observe that any map $C \to \widetilde{T_{Bg}}$ with compact domain $C$ can be homotoped to a map whose image lies in a copy of $\widetilde {BF_n}$, which is contractible.
\end{proof}


We will always view an $m\times n$-matrix $A$ over a ring $R$ as an $R$-homomorphism $R^m\to R^n$ by \emph{right}-multiplication since we prefer working with \emph{left}-modules.

For a monomorphism $g \colon F_n\to F_n$, let $G = F_n\ast_g$,  and let $s_1, \dots, s_n$ denote generators of $F_n$, and $t$ the stable letter of the HNN extension. The \emph{Fox matrix of $g$} is
\[F(g) = \left(\fox{g(s_i)}{s_j}\right)_{i,j = 1}^n\in {\Z F_n}^{n\times n}\]
 Put $\S = \{s_1,\dots,s_n,t\}$. We will often consider the matrix
\[ A(g; \S) =
\begin{pmatrix}
& & &   s_1 - 1 \\
& \operatorname{Id} - t\cdot F(g) &  &  \vdots  \\
&  &  & s_n-1
\end{pmatrix} \in \Z G^{n\times(n+1)}
\]
Given $s\in\S$, we let $A(g;\S, s)$ be the square matrix obtained from $A(g;\S)$ by removing the column which contains the Fox derivates with respect to $s$. Let $\Gamma_k = G/ G_r^{k+1}$, where $G_r^k$ are the subgroups of the rational derived series as introduced in \cref{rational derived}. Denote by $p_k\colon G\to \Gamma_k$ the projection and denote the ring homomorphisms $p_k\colon \Z G\to \Z\Gamma_k$ by the same letter. Notice that \[\Gamma_0 = H_1(G)_f =: H\]

The following theorem summarises the various invariants introduced in \cref{chapter: preliminaries} for descending HNN extensions of finitely generated free groups.

\begin{thm}\label{invariants for fbc}
	With the notation above, let $G = F_n\ast_g$ and $s\in\S$. Then:
	\begin{enumerate}
		\item\label{fbc:l2tor} For the universal $L^2$-torsion we have
			\[	\tor_u(G) =  -[\Z G^n\tolabel{A(g;\S, s)}\Z G^n] + [\Z G\tolabel{s-1}\Z G]\]
			and so
			\[  P_{L^2}(G) = P( \det_{\D(G)}(A(g;\S,s)) ) - P(s-1)\in \Pol_T(H)\]
		\item\label{fbc:higher alex} If $p_k(s)\neq 0$, then for the universal $L^2$-torsion relative to $p_k$ we have
			\[	\tor_u(G;p_k) =  -[\Z\Gamma_k^n\overset{p_k(A(g;\S,s))}{\longrightarrow} \Z\Gamma_k^n] + [\Z\Gamma_k\tolabel{p_k(s)-1}\Z\Gamma_k]\]
			and so
			\[P_{L^2}(G,p_k) = P( \det_{\D(\Gamma_k)}(p_k(A(g;\S,s)))) - P(p_k(s)-1)\in \Pol_T(H)\]
		\item\label{fbc:alex} In $\Wh^w(H) \cong (T^{-1} \Q H)^\times/\{ \pm h\mid h\in H\}$ we have
		\begin{equation*}
		\Delta_A(G) = \begin{cases}
			-\tor_u(G;p_0)  & \text{ if } b_1(G)\geq 2\\
			-\tor_u(G;p_0) \cdot (p_0(t)-1)  &\text{ if }  b_1(G)=1
		\end{cases}
		\end{equation*}
		\item\label{fbc:bns} Let $\phi\in\Hom(G,\R)$. If $\phi(s) \neq 0$, then $[-\phi]\in\Sigma(G)$ if and only if the map
		\[A(g;\S, s) \colon \widehat{\Z G}_\phi^n \to \widehat{\Z G}_\phi^n\]
		is surjective, or equivalently, bijective.
	\end{enumerate}
\end{thm}

\begin{proof}
	(\ref{fbc:l2tor}) We write the relations defining the descending HNN extension $G = F_n\ast_g$ as
	\[R_i = s_itg(s_i)^{-1}t^{-1}\]
	If we let $BF_n$ be the wedge of $n$ circles, then the $\Z G$-chain complex of the mapping torus $T_{Bg}$ has the form
	\[ C_* = 0\to\Z G^n\tolabel{c_2}\Z G^{n+1}\tolabel{c_1} \Z G\to 0 \]
	where $c_1$ is given by the transpose of
	\[\begin{pmatrix}
	s_1 - 1 & s_2 -1 & \dots & s_n -1 & t-1
	\end{pmatrix}
	\]
	and $c_2$ is given by the $n\times(n+1)$ matrix containing the Fox derivatives $\fox{R_i}{s_j}$ and $\fox{R_i}{t}$. This is precisely the matrix $A(g;\S)$ since
	
	\begin{align*}
	\fox{R_i}{s_j} &= \delta_{ij} + s_it \left(\fox{g(s_i)^{-1}}{s_j} + g(s_i)^{-1}\cdot \fox{t^{-1}}{s_j}\right) \\
	&= \delta_{ij} -s_itg(s_i)^{-1}\cdot\fox{g(s_i)}{s_j} \\
	&= \delta_{ij} - t \cdot\fox{g(s_i)}{s_j}\\
	\fox{R_i}{t} &= s_i - s_i tg(s_j)^{-1}t^{-1} = s_i -1
	\end{align*}
	where $\delta_{ij}$ denotes the Kronecker delta.
	
	Consider the $\Z G$-chain complexes
	\[ \xymatrix{
	B_* = &  0 \ar[r] &  0\ar[r]                  & \Z G \ar[r]^{s-1} & \Z G \ar[r] & 0\\
	D_* = &  0\ar[r]  & \Z G^n \ar[r]^{A(g;\S,s)} & \Z G^n\ar[r]      & 0 \ar[r]    & 0
	} \]
	We obtain a short exact sequence of $\Z G$-chain complexes
	\[ 0 \to B_*\to C_*\to D_* \to 0\]
	Since $B_*$ is $L^2$-acyclic by \cite[Theorem 3.14 (6) on page 129 and (3.23) on page 136]{Lueck2002}, $D_*$ is also $L^2$-acyclic and we have the sum formula \cite[Lemma 1.9]{FriedlLueck2015b}
	\begin{align*}
	\tor_u(G) & = \tor_u(C_*) \\
	& = \tor_u(B_*) + \tor_u(D_*)\\
	& =  [\Z G\tolabel{s-1}\Z G]  -[\Z G^n\tolabel{A(g;\S,s)}\Z G^n]
	\end{align*}	
	
	The statement
	\[P_{L^2}(G) = P( \det_{\D(G)}(A(g;\S,s)) ) - P(s-1)\in \Pol_T(H)\]
	is obtained by applying the polytope homomorphism $\P\colon \Wh^w(G)\to \Pol_T(G)$.
		
	\smallskip\noindent
	(\ref{fbc:higher alex}) This follows exactly as	(\ref{fbc:l2tor}) since the chain complex used to define $\tor_u(G;p_k)$ is
		\[ 0\to\Z \Gamma_k^n\tolabel{p_k(c_2)}\Z\Gamma_k^{n+1}\tolabel{p_k(c_1)} \Z \Gamma_k\to 0\]
		
	\smallskip\noindent
	(\ref{fbc:alex}) A $\Z H$-presentation of the Alexander module $A(G)$ is given by
		\[ \Z H^n \tolabel{p_0(A(g;\S))}\Z H^{n+1}\to A(G)\to 0\]
		We now apply the same argument as in the proof of \cite[Theorem 5.1]{McMullen2002}: If $b_1(G)\geq 2$, then this yields
		\[ \det(p_0(A(g;\S, s))) = (p_0(s)-1)\cdot \Delta_A(G)\]
		for all $s\in\S$ such that $p_0(s) \neq 0$.
		
		If $b_1(G) =1$, then
		\[ \det(p_0(A(g;\S, t))) = \Delta_A(G)\]
		Since the isomorphism $\Wh^w(G)\cong T^{-1}(\Z H)$ is given by the determinant over $T^{-1}(\Z H)$, the claim follows from part (\ref{fbc:higher alex}) for $k=0$ (since $\Gamma_0 = H$).

	\smallskip\noindent
	(\ref{fbc:bns})	By \cref{bns criterion}, $[-\phi]\in\Sigma(G)$ if and only if \[H_0(G; \widehat{\Z G}_\phi) =0 \textrm{ and } H_1(G; \widehat{\Z G}_\phi )=0\]
	The chain complex computing these homology groups is
	\[  0\to\widehat{\Z G}_\phi^n\tolabel{c_2}\widehat{\Z G}_\phi^{n+1}\tolabel{c_1} \widehat{\Z G}_\phi\to 0\]
	We assume $\phi(s) \neq 0$ for a fixed $s\in\S$. Since $G$ is locally indicable (by \cref{G is locally indicable}), \cref{invertible} shows that $s-1$ is invertible in $\widehat{\Z G}_\phi$, which implies that $c_1$ is surjective, and therefore $H_0(G; \widehat{\Z G}_\phi) =0$ for any non-zero $\phi$.
	
	Assume without loss of generality that $s = s_1$. Then the kernel of $d_1$ is the set
	\[ K = \left\{(x_1,\dots, x_{n+1})\in \widehat{\Z G}_\phi^{n+1} \mid \sum_{k=2}^{n+1} x_k (s_k-1) (s_1-1)^{-1}= -x_1 \right\}\]
	By forgetting the first coordinate we see that $K$ is $\widehat{\Z G}_\phi$-isomorphic to $\widehat{\Z G}_\phi^{n}$, and
	\[H_1(G; \widehat{\Z G}_\phi )=0\]
	is equivalent to
	\[A(g;\S,s) \colon \widehat{\Z G}_\phi^{n} \to \widehat{\Z G}_\phi^{n}\]
	being surjective.
	
	Since $\hat{\Z G}_\phi$ is stably finite (this was shown by Kochloukova~\cite{Kochloukova2006}), an epimorphism $\hat{\Z G}_\phi^n \to \hat{\Z G}_\phi^n$ is necessarily an isomorphism.
\end{proof}

\begin{rmk}\label{non-degenerate}
		Note that the above proof shows (and uses) that $A(g;\S,s)$ (resp. $p_k(A(g;\S,s)$) is invertible over $\D(G)$ (resp. $\D(\Gamma_k)$). We will henceforth call a $\Z G$-square matrix with this property \emph{non-degenerate}.
\end{rmk}

\begin{ex}\label{example:l2tor}
	Using part (\ref{fbc:l2tor}) of the above theorem we compute the $L^2$-torsion polytope in a few examples. We use $a,b, c,\dots$ to denote some fixed generators of $F_n$.
	\begin{enumerate}
		\item For arbitrary $n$ and $g = \text{id}$ the polytope is just a line of length $n-1$ between $0$ and $t^{n-1}$.
		\item For $g\colon F_2\rightarrow F_2, \; x\mapsto a^kxa^{-k}$ for some $k\in \Z$, we get a tilted line between $0$ and $a^kt$.
		\item For $g\colon F_3\rightarrow F_3,\; a\mapsto b, \: b \mapsto c, \: c\mapsto a[b,c]$ we get a triangle as shown below.
		\begin{figure}[h]
			\begin{tikzpicture}[baseline=2.6mm, x=.8cm, y=.8cm, domain=-1:1]
			\draw[thick] (0,0) -- (0,1);
			\draw (0,0) node{$\bullet$} node[right]{$0$};
			\draw (0,1) node{$\bullet$} node[right]{$t^{n-1}$};
			\end{tikzpicture} \hspace{8mm}
			\begin{tikzpicture}[baseline=2.6mm, x=0.8cm, y=.8cm, domain=-1:1]
			\draw[thick] (0,0) -- (3,1);
			\draw (0,0) node{$\bullet$} node[left]{$0$};
			\draw (3,1) node{$\bullet$} node[right]{$a^kt$};
			\end{tikzpicture} \hspace{8mm}
			\begin{tikzpicture}[baseline=-1mm, x=.6cm, y=.6cm, domain=-1:1]
			\draw[thick, fill = purple] (0,1) -- (0,-1) -- (2,0) -- (0,1);
			\draw (0,1) node{$\bullet$} node[left]{$t^2$};
			\draw (0,-1) node{$\bullet$} node[left]{$0$};
			\draw (2,0) node{$\bullet$} node[right]{$a^2t$};
			\end{tikzpicture}
			\caption{The $L^2$-torsion polytopes in \cref{example:l2tor}}
		\end{figure}
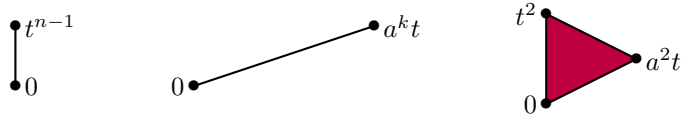
	\end{enumerate}
\end{ex}

More importantly, we can now show that the $L^2$-torsion polytope of free group HNN extensions induces indeed a seminorm on the first cohomology.

\begin{cor}\label{indeed seminorm}
	Let $G = F_n\ast_g$. Then the Thurston seminorm
	\[\|\cdot \|_T\colon H^1(G;\R)\to \R\]
	as defined in \cref{def:Thurston for fbc} is indeed a seminorm.
\end{cor}
\begin{proof}
	As a difference of seminorms it is clear that $\|\cdot \|_T$ is $\R$-linear and continuous.
	
	First let $\phi\in H^1(G;\Q)$ be a rational class.
	We easily find a generating set $s_1, \dots , s_n$ of $F_n$ such that $\phi(s_1)=0$. We add a stable letter to this set, and form a generating set $\S$ for $G$.
	
	We get from the previous theorem
	\[      \tor_u(G) =  -[\Z G^n\tolabel{A(g;\S, s_1)}\Z G^n] + [\Z G\tolabel{s_1-1}\Z G]\]
	
	By \cite[Theorem 2.2]{FriedlHarvey2007} of Friedl--Harvey applied to $\mathbb{K} = \D(K)$, the polytope $P(\det_{\D(G)}(A(g;\S, s_1)))$ defines a seminorm on $H^1(G;\R)$ which we denote by $\|\cdot\|_{T'}$. Then, since $\phi(s_1) = 0$, we have
	\[	\|\phi\|_T = \|\phi\|_{T'} \geq 0\]
	and for any $\psi\in H^1(G;\R)$
	\begin{align*}
		\|\phi+\psi\|_T &=  \|\phi +\psi\|_{T'} - |(\phi +\psi)(s_1)| \\
									&\leq \|\phi\|_{T'} + \|\psi\|_{T'} - |\psi(s_1)|\\
									&= \|\phi\|_T + \|\psi\|_T
	\end{align*}
	This finishes the proof for rational classes.
	
	The general case directly follows by the continuity of $\|\cdot\|_T$.
\end{proof}

\section{Thurston, Alexander and higher Alexander norms}\label{sec: inequalities}

In this section we are going to extend the inequalities between the Alexander norm, the higher Alexander norms of Harvey, and the Thurston norm from the setting of $3$-manifolds to that of free-by-cyclic groups. Specifically, 
we will prove an analogue of McMullen's~\cite[Theorem 1.1]{McMullen2002} and Harvey's~\cite[Theorem 10.1]{Harvey05} for the newly defined Thurston norm of descending HNN extensions of free groups.

The key technical tool used is the notion of a $\Phi$-specialising group, introduced in \cref{subsec: specialising}.

\begin{prop}
\label{spec then good}
Let $R$ be a ring, and let $R_s[z^\pm]$ be a ring of twisted Laurent polynomials determined by an automorphism $s\colon R\to R$. Let $\mathcal D$ and $\mathcal D'$ be skew-fields and $t\colon \D\to\D, \: t'\colon \D'\to\D'$ automorphisms. Let $\beta\colon R \to \mathcal D$ and $\beta'\colon R \to \mathcal D'$ be two $R$-fields such that $\beta\circ s =  t \circ\beta$ and $\beta'\circ s =  t' \circ\beta'$. Suppose that there is a specialisation $(S,\sigma)$ from $\D$ to $\D'$, with $S$ preserved by $t$ and $t' \circ \sigma = \sigma \circ t$.
Then for any square matrix $M$ over $S_t[z]$, we have
\[
 \deg \det_{\D} M \otimes \D \geqslant \deg \det_{\D'} M\otimes \D'
\]
where $\deg$ denotes the degree of Laurent polynomials in $z$.
\end{prop}
\begin{proof}
We are going to prove the desired inequality by a triple induction. Firstly we induct on the size of the matrix $M$; secondly, on the number of non-zero entries in the first column of $M = (m_{ij})$; thirdly on the
sum  $d$ of the degrees of the elements of the first column.

For $1 \times 1$ matrices the result follows trivially, since
\[
 \deg m_{11} \geqslant \deg \sigma(m_{11})
\]
as the support of the Laurent polynomial $\sigma(x)$ is contained in the support of the Laurent polynomial $x$ for any $x \in S_t[z]$.

Now suppose that $M$ is an $n \times n$ matrix with $n>1$.
If the first column of $M$ is trivial, then $\det_\D M\otimes \D = 0$ and $\det_{\D'} M\otimes \D' = 0$, and so the degrees are both equal.

When the first column is not trivial, we need to consider two cases.
Firstly, there might be only one non-zero entry in the first column of $M$. Then both determinants are products of determinants of the same smaller matrix (taken over $\D$ and $\D'$), and an element in $S_t[z]$. In this case we are done by the induction hypothesis.

Secondly, there might be more than one non-trivial entry in the leftmost column of $M$.
Again, we need to consider two situations. Suppose first that the lowest and highest terms appearing in any non-zero $m_{i1}$ are not trivialised by $\sigma$. Then we can perform the first step of Euclid's algorithm using an elementary matrix $E$ whose off-diagonal entry lies in $S_t[z]$ -- it is the product of the lowest term of one entry and the inverse of the lowest term of another entry in the first column. Therefore $E$ and $EM$ are matrices over $S_t[z]$, and we have 
\[\det_{\D} M \otimes \D = \det_{\D} (E \otimes \D \cdot  M \otimes \D)\]
and 
\[ \det_{\D'} M\otimes \D' = \det_{\D'} (E\otimes \D' \cdot M\otimes \D') \]
The sum of the degrees of the elements of the first column of $EM$ is lower than that of $M$, and the number of non-trivial entries was not increased.

The second possibility occurs when one of the entries $m_{i1}$ has a lowest or highest term with coefficient being mapped to $0$ by $\sigma$. Suppose that this term is $x z^k$. Without loss of generality we may assume that it occurs in $m_{11}$. Let $M'$ be obtained from $M$ by subtracting $x z^k$ from $m_{11}$, and $M''$ be obtained from $M$ by forcing the first column to be made of zeroes, except for the first entry which is made equal to $x z^k$. (In this case, Cohn calls $M$ the \emph{determinantal sum} of $M'$ and $M''$, for reasons which will become apparent below.) We have
\[
 \det_\D^c M \otimes \D = \det_\D^c  M'\otimes \D  + \det_\D^c M''\otimes \D 
 \]
and
\[
 \det^c_{\D'} M\otimes\D' = \det^c_{\D'} M'\otimes\D' + \det^c_{\D'} M''\otimes\D' =  \det^c_{\D'} M'\otimes\D'
\]
since $M''\otimes\D'$ has a column of zeroes.
By induction
\[\deg \det_\D M'\otimes\D \geqslant \deg \det_{\D'} M'\otimes\D'\]

The coefficients of $\det^c_\D M''\otimes \D$ are all mapped to $0$ by $\sigma$, since they are all multiples of $x$.
It is now clear that the set of powers of $z$ with a coefficient not being mapped to $0$ by $\sigma$ is the same in $\det_\D^c M$ and in $\det^c_{\D'} M'\otimes \D$; but this is precisely the set of powers which are still visible in $\det^c_{\D'} M\otimes \D'$. This proves the claim.
\end{proof}

We will use the above proposition in two ways: firstly, it will allow us to show that any descending HNN extension $G$ of a free group is $H^1(G;\Z)$-specialising; secondly, we will use it directly to prove the inequality between Thurston and higher Alexander norms for $G$.

\begin{rmk}
 Let $\Phi$ be a family of morphisms $G \to \R$, and let $H$ be a subgroup of $G$. Then $\Phi$ forms naturally a family of morphisms $H \to \R$. Moreover, if a morphism takes $G$ to $\Z$ then it also takes $H$ to $\Z$.
\end{rmk}

\begin{cor}
\label{extensions are spec}
Let $G$ be a group.
Let $\phi \colon G \to \Z$ be an epimorphism with kernel $K$, and let $\Phi$ be a collection of homomorphisms $G\to \R$. If $K$ is $\Phi$-specialising then $G$ is $(\Phi \cup \{\phi\})$-specialising.
\end{cor}
\begin{proof}
We start by remarking that $G$ is torsion free and satisfies the Atiyah conjecture, since $K$ does and $\Z$ is torsion free and elementary amenable (see \cite[Theorem 10.19]{Lueck2002}).

Choose an element $z \in \phi^{-1}(1)$ in $G$.
 Let $R = \Q K$. Recall from \cref{crossed product of extension}  that $\Q G$ has the structure of a twisted Laurent polynomial ring over $R$ with variable $z$.

 Let $\alpha \colon G \to \Gamma$ be an epimorphism to a torsion-free elementary amenable group $\Gamma$ such that every  morphism $\psi \in \Phi \cup \{\phi\}$ factors through $\alpha$. Let $L = \alpha(K)$. Since $\phi$ factors through $\alpha$, the ring $\Q \Gamma$ has the structure of a twisted Laurent polynomial ring over $\Q L$ with variable $\alpha(z)$; we will abuse the notation and call this variable $z$ as well. This way $\alpha\vert_K$ is a $z$-equivariant map.

 Note that $L$ is torsion-free and elementary amenable, and that every element in $\Phi$ restricted to $K$  factors through $\alpha\vert_K$. Thus, by assumption on $K$, there exists a specialisation $(S,\sigma)$ from the epic $R$-field $\D(K)$ to the epic $R$-field $\D(L)$, where the map $R = \Q K \to \Q L \to \D(L)$ is induced by $\alpha\vert_K$. Note that the maps $\Q K \to D(K)$ and $\Q K \to D(L)$ are $z$-equivariant -- see \cite[Lemma 10.57]{Lueck2002}. We may also require that $S$ is preserved by the $z$-action -- e.g. we may replace $S$ by $\bigcap_{k\in \Z} z^k(S)$; it is immediate that $\sigma$ will be $z$-equivariant as well.

 We now claim that there exists a specialisation from $\D(G)$ to $\D(\Gamma)$. In view of \cref{spec criterion}, let $M$ be a square matrix over $\Q G$ such that $M\otimes \D(\Gamma)$ is invertible. We can view $M$ as a matrix over the Laurent polynomial ring $R_t[z]$, and $\alpha(M) = M\otimes \D(\Gamma)$ as a matrix over the polynomial ring $(\Z L)_t[z]$. Since $M\otimes \D(\Gamma)$ is invertible, we have
 \[
  \det_{\D(\Gamma)} M\otimes \D(\Gamma) \neq 0
 \]
and hence
\[
 \deg \det_{\D(\Gamma)} M\otimes \D(\Gamma) \geqslant 0
\]
Now we apply \cref{spec then good} and conclude that
\[
 \deg \Ddet M \geqslant 0
\]
which implies that $\Ddet M \neq 0$, and so $M\otimes \D(G)$ is invertible. This proves the claim.
\end{proof}

Recall that $b_1(G)$ denotes the (usual) first Betti number of $G$.

\begin{thm}
\label{main thm alex vs thurston}
 Let $G = F_n \ast_g$ be a descending HNN extension of $F_n$ with stable letter $t$, and let $\psi \in H^1(G;\R)$. Then
 \[
  \delta_1(\psi) \leqslant \delta_2(\psi) \leqslant \dots \leqslant \| \psi \|_T
 \]
If $b_1(G) \geqslant 2$, then also $\delta_0(\psi) \leqslant \delta_1(\psi)$.
If $b_1(G) = 1$, then $\delta_0(\psi) -| \psi(t) | \leqslant \delta_1(\psi)$.

When $\psi$ is fibred (that is $\ker \psi$ is finitely generated), then all the inequalities above become equalities.
\end{thm}
\begin{proof}
We start by noting that it is enough to verify the statements for integral classes; once this is done, the statements for rational classes follow immediately, and for general classes follow from continuity of the norms.

Since $G$ is a finitely presented group of deficiency at least $1$, Harvey showed in \cite[Corollary 2.3]{Harvey06} that we have $\delta_i(\psi) \leqslant \delta_{i+1}(\psi)$ for every $i>0$. She also proved the inequalities involving $\delta_0$. Thus we need only show that $\delta_i(\psi) \leqslant \| \psi \|_T$ for $i>0$. To this end, pick such an $i$.

Let $\phi \colon G \to \Z$ denote the canonical epimorphism induced by the HNN extension.  Let $\psi \colon G \to \Z$ be a homomorphism.
Let $p_i \colon G \to \Gamma_i$ denote the map associated to $\delta_i$. Note that $\Gamma_i$ maps onto $\Gamma_0$, which is the free part of the abelianisation of $G$. Let $K = \ker \psi$. If $\psi = \pm \phi$, then $K$ is locally free (since $G$ is a descending HNN extension). By \cref{loc free spec}, $K$ is specialising.

If $\psi \neq \pm \phi$, then $\phi|_K$ is non-trivial. It is immediate that $\phi|_K \colon K \to \Z$ gives $K$ the structure of a (locally-free)-by-cyclic group. By \cref{loc free spec,extensions are spec}, $K$ is $\{\phi\vert_K\}$-specialising.

Recall that $\delta_i = \norm (P_{L^2}(G,p_i))$ and $\| \cdot \|_T = \norm (P_{L^2}(G))$. \cref{invariants for fbc}(\ref{fbc:higher alex}) tells us that
\[
 P_{L^2}(G,p_i) = P(\det_{\D(\Gamma_i)} (p_i(A(g;\S,s)))) - P(p_i(s) - 1)
\]
and
\[
 P_{L^2}(G) = P(\Ddet (A(g;\S,s))) - P(s - 1)
 \]
where $\S$ is a generating set of $G$, and $s \in \S$ is such that $p_i(s) \neq 0$.
Recall that $\Q G$ is naturally a twisted Laurent polynomial ring $\Q K_t[z^\pm]$. Let $L = \ker (\psi \colon \Gamma_i \to \Z)$, and consider the subrings $\D(K)_t[z^\pm]\subseteq \D(G)$ and $\D(L)_t[z^\pm]\subseteq \D(\Gamma_i)$. Since $K$ is $\{\phi|_K\}$-specialising and $\phi|_K$ factorises over $p_i|_K\colon K\to L$, there exists a specialisation from $\D(K)$ to $\D(L)$.
We now apply \cref{spec then good} with $R = \Q K,  \D = \D(K)$, and $\D' = \D(L)$ and obtain 

\begin{align*}
 \delta_i(\psi) &= \deg \det_{\D(\Gamma_i)} (p_i(A(g;\S,s))) - \deg (p_i(s) - 1) \\
  &= \deg \det_{\D(\Gamma_i)} (p_i(A(g;\S,s))) - | \psi(s) | \\
  &= \deg \det_{\D(\Gamma_i)} (p_i(A(g;\S,s))) - \deg (s - 1) \\
  &\leqslant \deg \Ddet (A(g;\S,s)) - \deg (s - 1) \\
  &= \| \psi \|_T
\end{align*}
where the degrees are taken of Laurent polynomials in $z$. Note that to use \cref{spec then good} we need to guarantee that the embedding $\Q K \into \D(K)$ and the map $\Q K \to \Q L \into \D(L)$, as well as the specialisation, are $z$-equivariant, but this follows from \cite[Lemma 10.57]{Lueck2002} and a discussion as before.

\smallskip
Now suppose that $\psi$ is fibred, that is that $K = \ker \psi$ is finitely generated. It follows from the work of Geoghegan--Mihalik--Sapir--Wise~\cite[Theorem 2.6 and Remark 2.7]{Geogheganetal2001} that $K$ is finitely generated free itself, say of rank $m$. Denote the inclusion by $i\colon K\to G$.

        By claim (3.26) made in the proof of \cite[Theorem 3.24]{FriedlLueck2015b}, we have
                \[ \|\psi\|_T = \norm(P_{L^2}(G))(\psi) = \norm(\P(-\tor_u(G)))(\psi) = -\chi^{(2)}(i^*\widetilde{T};\mathcal{N}(K))
                \]
            where $T$ is the mapping telescope of a realisation of $g$. Recall that the $K$-CW-complex $i^*\widetilde{T}$ is a model for $EK$ and that $K$ is finitely generated free, so
            \[ -\chi^{(2)}(i^*\widetilde{T};\mathcal{N}(K)) = -\chi^{(2)}(K) = b_1^{(2)}(K) = b_1(K) - 1\]
        
        McMullen showed in \cite[Theorem 4.1 and Theorem 5.1]{McMullen2002} that we have
        \[ b_1(\ker\psi)-1 \leqslant \| \psi \|_A \]
        irrespective of the fact that $\psi$ is fibered. (In fact, McMullen showed that this is an equality when $\psi$ lies in the cone over an open face of the unit ball of the Alexander norm.)

        Combining the above results  with \cref{main thm alex vs thurston} we obtain
        \[ b_1(\ker\psi)-1 \leqslant \| \psi \|_A \leqslant \|\psi\|_T = b_1(\ker\psi)-1 \qedhere\]
\end{proof}

\begin{rmk}
Dunfield in~\cite{Dunfield2001} constructed a hyperbolic $3$-manifold which fibres, and whose Thurston and Alexander norms do not agree. His example is actually a link complement, and thus a manifold with toroidal boundary. Since it does fibre, it must do so over a surface with a non-empty boundary. Thus the fundamental group of the $3$-manifold is a free-by-cyclic group, and hence noting that our definition of the Thurston norm coincides with the usual one for a $3$-manifold (as shown in \cite[Theorem 3.27]{FriedlLueck2015b}), we conclude that Dunfield's example shows that also in our setting the Alexander and Thurston norms are not equal in general.
\end{rmk}

\section{The \texorpdfstring{$L^2$}{L\texttwosuperior}-torsion polytope and the BNS-invariant}

	In this section we relate the $L^2$-torsion polytope of a descending HNN extension of $F_2$ with the BNS-invariant introduced in \cref{sec:bns}. This approach is motivated by the following results: If $M$ is a compact orientable $3$-manifold, the unit norm ball of the Thurston norm is a polytope, and there are certain maximal faces such that a cohomology class comes from a fibration over the circle if and only if it lies in the positive cone over these faces \cite{Thurston1986}. Bieri-Neumann-Strebel \cite[Theorem E]{Bierietal1987} showed that the BNS-invariant $\Sigma(\pi_1(M))$ is precisely the projection of these \emph{fibered faces} to the sphere $S(G) = (\Hom(G,\R)\s- \{0\})/\R_{>0}$. Since the $L^2$-torsion polytope induces the Thurston norm for descending HNN extensions of $F_n$, we expect a similar picture in this setting. The work of Friedl-Tillmann \cite[Theorem 1.1]{FriedlTillmann2015} provides further evidence for this expectation.

\begin{dfn}
Let $H$ be an abelian group with a total ordering $\leqslant$, which is invariant under multiplication. Let $R$ be a skew-field. We define
$R ( H, \leqslant)$ to be the set of functions $H \to R$ with well-ordered support, that is $f \colon H \to R$ belongs to $R ( H, \leqslant)$ if every subset of $H$ whose image under $f$ misses zero has a $\leqslant$-minimal element.
\end{dfn}

\begin{thm}[Malcev, Neumann~\cite{Malcev1948,Neumann1949}]
Convolution is well-defined on \[R ( H, \leqslant)\] and turns it into a skew-field.
\end{thm}

\begin{rmk}\label{Malcev Neumann cross}
 In fact, given structure maps $\phi \colon H \to \Aut(R)$ and $\mu \colon H\times H\to R^\times$ of a crossed product $R*H$, one can also define a crossed-product convolution on $R ( H, \leqslant)$ in a way completely analogous to the usual construction of crossed product rings (see \cref{def crossed product}). The resulting ring is still a skew-field, and we will denote it by $R \ast ( H, \leqslant)$ for emphasis.
\end{rmk}

\begin{rmk}
 In fact the Malcev--Neumann construction works for all biorderable groups, and not merely abelian ones.
\end{rmk}

In order to relate the $L^2$-torsion polytope to the BNS-invariant, we first need to put the skew-field $\D(G)$ and the Novikov-Sikorav completion $\widehat{\Z G}_\phi$ (introduced in \cref{def:sikorav-novikov}) under the same roof.

\begin{lem}
        Let $K = \ker(p_0\colon G\to \Gamma_0 = H_1(G)_f)$. Given $\phi\in\Hom(G,\R) \s- \{0\}$ with $L = \ker(\phi)$, let $\leq_\phi$ be a multiplication invariant total order on $H_1(G)_f$ such that $\phi$ is order-preserving (we endow $\R$ with the standard ordering $\leqslant$). We define
        \[ \mathfrak{F}(G,\phi) := \D(K) * (H_1(G)_f,\leq_\phi)\]
        in the sense of \cref{Malcev Neumann cross}. Then there is a commutative diagram of rings
        \[
        \xymatrix{
                \Z K * H_1(G)_f\ar[r] &   \D(K) * H_1(G)_f\ar[r] & \D( G) \ar[d]^{i_\phi}\\
                \Z G\ar[u]^\cong\ar[d]_\cong\ar@/_/[rru] \ar@/^/[rrd] &  &   \mathfrak{F}(G,\phi)\\
                \Z L *\im \phi \ar[r] & \widehat{\Z L *  \im \phi}_{\iota} \ar[r]^\cong & \hat{\Z G}_\phi \ar[u]_{j_\phi}
                }
        \]
        such that all maps are inclusions, where $\iota$ denotes the inclusion $\im \phi \into \R$, and
        $\widehat{\Z L *  \im \phi}_{\iota}$
        denotes the Sikorav--Novikov completion of $\Z L *  \im \phi$ with respect to $\iota \colon \im \phi \to \R$.
\end{lem}
\begin{proof}
        All maps apart from $i_\phi$ and $j_\phi$ are either obvious or have already been explained. The commutativity of the upper and lower triangle is clear.

        Since $\mathfrak{F}(G,\phi)$ is a skew-field, the universal property of the Ore localisation allows us to define
        \[i_\phi\colon \D(G) \cong T^{-1}(\D(K) * H_1(G)_f)\to \mathfrak{F}(G,\phi)\]
        as the localisation of the obvious inclusion
        \[ \D(K) * H_1(G)_f\to \mathfrak{F}(G,\phi)\]

        The definition of
        \[j_\phi\colon \hat{\Z G}_\phi\cong \widehat{\Z L * \im \phi}_{\iota} \to \mathfrak{F}(G,\phi)\]
        uses the same formulae as the composition
        \[ \Z L *\im \phi \overset{\cong}{\longrightarrow} \Z G \overset{\cong}{\longrightarrow}      \Z K * H_1(G)_f\]
        and we need to verify that this indeed maps to formal sums with well-ordered support with respect to $\leq_\phi$. But this follows directly from the fact that
        \[\phi\colon H_1(G)_f \to \R \]
        is order-preserving. The commutativity of the right-hand triangle follows immediately.
\end{proof}

\begin{dfn}
        Given $\phi\in\Hom(G,\R)$ and
        \[ x= \!\!\!\!\!\! \sum_{h\in H_1(G)_f} \!\!\!\!\!\! x_h\cdot h\in \D(K) * (H_1(G)_f, \leq_\phi)\]
        we set
        \[ S_\phi(x) = \minsupp_\phi(x) = \big\{h\in\supp(x)\mid \phi(h) = \min\{ \phi(\supp(x))\}\big\}\]
        and define $\mu_\phi\colon \mathfrak{F}(G,\phi)^\times\to \mathfrak{F}(G,\phi)^\times$ by
        \[
                \mu_\phi\big( \!\!\!\!\!\! \sum_{h\in H_1(G)_f} \!\!\!\!\!\! x_h\cdot h  \big) = \sum_{h\in S_\phi(x)} x_h\cdot h
        \]
\end{dfn}

We record the following properties.

\begin{lem}\label{mu}
        Let $\phi\in\Hom(G,\R)$.
        \begin{enumerate}
                \item\label{mu:homo} The map $\mu_\phi$ is a group homomorphism.
                \item\label{mu:restrict} It restricts to maps (denoted by the same name)
                \begin{gather*}
		                \mu_\phi\colon \D(G)^\times \to\D(G)^\times\\
                        \mu_\phi\colon \hat{\Z G}_\phi^\times\to \Z G \s- \{0\}
                \end{gather*}
                and the latter map agrees with $\mu_\phi\colon \hat{\Z G}_\phi^\times\to \Z G \s- \{0\}$ from \cref{alpha}.
        \end{enumerate}
\end{lem}
\begin{proof}
	This is obvious.
\end{proof}

We now give a practical method for calculating the BNS invariant for descending HNN-extensions of $F_2$.

\begin{thm}
\label{formula for BNS}
Let $G$ be a descending HNN extension of $F_2$. Let \[\phi \in\Hom(G,\R) \s- \{0\}\] Suppose that $x,y$ are generators of $F_2$ for which $\phi(x), \phi(y) > 0$, and let $g \colon F_2\to F_2$ be a monomorphism such that $G = F_2*_g$, and such that $g(x),g(y)$ have no common prefix.
Then $[-\phi] \in \Sigma(G)$ \iff \[\mu_\phi(1 + t \fox{g(x)}{y} - t\fox{g(y)}{y}) = \pm z\] for some $z \in G$.
\end{thm}
\begin{proof}
By \cref{invariants for fbc}, (\ref{fbc:bns}), we have $-\phi\in\Sigma(G)$ if and only if the map
\[A \colon {\hat{\Z G_\phi}}^2 \to {\hat{\Z G_\phi}}^2\]
        is an isomorphism, where
        \[ A = A(g;\S,x) = \left( \begin{array}{cc}
        -t \fox{g(x)}{y} & x-1 \\
        1-t\fox{g(y)}{y} & y-1
        \end{array} \right)  \]

         Since $\phi(y) \neq 0$, the element $y-1$ is invertible in $\hat{\Z G_\phi}$, and thus we may perform an elementary row operation over $\hat{\Z G_\phi}$ to obtain a triangular $\hat{\Z G_\phi}$-matrix
    \[ B = \left( \begin{array}{cc}
        -t \fox{g(x)}{y} - (1-t\fox{g(y)}{y})(y-1)^{-1}(x-1) & 0 \\
        1-t\fox{g(y)}{y} & y-1
        \end{array} \right) \]
        Note that $A$ is invertible over $\hat{\Z G_\phi}$ if and only if the diagonal entries of $B$ are invertible in $\hat{\Z G_\phi}$.
        One of the diagonal entries is $y-1$, which we already know to be invertible. The other one is invertible \iff
        \[ \mu_\phi\big(-t \fox{g(x)}{y} - (1-t\fox{g(y)}{y})(y-1)^{-1}(x-1)\big) = \pm z\]
        for some $z\in G$, thanks to \cref{invertible}. But
        \[ \mu_\phi\big((1-t\fox{g(y)}{y})(y-1)^{-1}(x-1)\big) = \mu_\phi(1-t\fox{g(y)}{y})\]
        and the supports of $1-t\fox{g(y)}{y}$ and $t \fox{g(x)}{y}$ have a trivial intersection: the lack of common prefixes of $g(x)$ and $g(y)$ implies that the only element in $G$ which could lie in both supports is $t$, but then we would need to have both $g(x)$ and $g(y)$ starting with $y$, which would yield a non-trivial common prefix.

        This implies
        \begin{align*}
         & \phantom{{}={}} \mu_\phi\big(-t \fox{g(x)}{y} - (1-t\fox{g(y)}{y})(y-1)^{-1}(x-1)\big)\\
         &= \mu_\phi\big(-t \fox{g(x)}{y} - 1+t\fox{g(y)}{y}\big) \qedhere
        \end{align*}
\end{proof}

\begin{rmk}
The above theorem does not apply to $\phi \in H^1(G;\R) \s- \{0\}$ which have $F_2 \leqslant \ker \phi$. There are however only two such cohomology classes (up to scaling): $\psi$, the class induced by the HNN-extension $G = F_2 \ast_g$, which lies in $\Sigma(G)$ \iff $g \colon F_2 \to F_2$ is an isomorphism, and $-\psi$, which always lies in $\Sigma(G)$.

For every other $\phi \in H^1(G;\R) \s- \{0\}$ one easily finds appropriate generators $x$ and $y$, and then any monomorphism $F_2 \to F_2$ inducing $G$ can be made into the desired form by postcomposing it with a conjugation of $F_2$. Such a postcomposition does not alter the isomorphism type of $G$.
\end{rmk}

Next we are going to relate the $L^2$-torsion polytope $P_{L^2}(G)$ to the BNS invariant for $G = F_2*_g$. For this we need some more preparations.

\begin{dfn}\label{face map}
	Let $H$ be a finitely generated free-abelian group. Let $P\subseteq H\otimes_\Z\R$ be a polytope and take $\phi\in\Hom(H,\R)$. We define the \emph{minimal face} of $P$ for $\phi$ to be
	\[F_\phi(P) = \{p\in P\mid \phi(p) = \min\{ \phi(q)\mid q\in P\} \}\]
	It is easy to see that $F_\phi$ respects Minkowski sums and hence induces group homomorphisms
	\begin{gather*}
		F_\phi\colon \Pol(H) \to\Pol(H)\\
		F_\phi\colon \Pol_T(H) \to\Pol_T(H)\\
	\end{gather*}
\end{dfn}

\begin{dfn}
\label{polytope equivalent}
	Let $K = \ker (p_0\colon G\to H_1(G)_f =: H)$, and let $x \in \D(G) = T^{-1}(\D(K) \ast H)$ and $\phi,\psi\in\Hom(G,\R) = \Hom(H,\R)$. We call $\phi$ and $\psi$ \emph{$x$-equivalent} if we can write $x = u^{-1}v$ with $u,v \in \D(K) \ast H$ in such a way that
	\[F_\phi(P(u)) = F_\psi(P(u)) \text{ and } F_\phi(P(v)) = F_\psi(P(v))\]

\end{dfn}

\smallskip

We are aiming at proving that the universal $L^2$-torsion determines the BNS-invariant for descending HNN extensions of free groups. In this process the following lemma is crucial in order to extract algebraic information about Dieudonné determinants from geometric properties of their polytopes.

\begin{lem}\label{polytope and minimal}
	Let $x\in\D(G)^{\times}$ and $\phi, \psi\in\Hom(G,\R)$.
	If $\phi$ and $\psi$ are $x$-equivalent, then
	\[      \mu_\phi(x) = \mu_\psi(x)\]
\end{lem}
\begin{proof}
	Write $x = u^{-1}v$ with $u,v\in \D(K)* H_1(G)_f$, so that by assumption we have
	\[F_\phi(P(u)) = F_\psi(P(u)) \text{ and } F_\phi(P(v)) = F_\psi(P(v))\]
	
	But $F_\phi(P(u)) = F_\psi(P(u))$ implies
	\[ \minsupp_\phi(u) = \minsupp_\psi(u)\]
	and so
	\[\mu_\phi(u) = \mu_\psi(u)\]
	The same argument applies to $v$ and so the claim follows from
	\begin{equation*}\label{xts}
	\mu_\phi(x) = \mu_\phi(u)^{-1}\cdot \mu_\phi(v) \qedhere
	\end{equation*}
\end{proof}

The following is similar to \cite[Theorem 1.1]{FriedlTillmann2015}; although we do not provide markings on the polytopes which fully detect the BNS-invariant, \cref{formula for BNS} makes up for this lack. The crucial point now is that the BNS invariant is locally determined by a polytope.

\begin{thm}
\label{main thm BNS}
        Let $g \colon F_2\to F_2$ be a monomorphism and let $G = F_2*_g$ be the associated descending HNN extension. Given $\phi \in\Hom(G,\R) \s- \{0\}$ such that $-\phi$ is not the epimorphism induced by $F_2 \ast_g$, there exists an open neighbourhood $U$ of $[\phi]$ in $S(G)$ and an element $d\in\D(G)^\times$ such that:
        \begin{enumerate}
        	\item The image of $d$ under the quotient maps
		        \begin{equation*}
		        	\D(G)^\times \to \D(G)^\times / [\D(G)^\times,\D(G)^\times] \cong K_1^w(\Z G) \to \Wh^w(G)
	        	\end{equation*}
	        	is $-\tor_u(G)$. In particular $P_{L^2}(G) = P(d)$ in $\Pol_T(H_1(G)_f)$.
	        \item For every $\psi, \psi' \in\Hom(G,\R) \s- \{0\}$ which satisfy $[\psi], [\psi'] \in U$ and are $d$-equivalent, we have $[-\psi] \in \Sigma(G)$ if and only if $[-\psi'] \in \Sigma(G)$.
		 \end{enumerate}
\end{thm}
\begin{proof}
	Suppose that $\ker \phi \neq F_2$. We easily find generators $x,y$ of $F_2$ for which $\phi(x), \phi(y)>0$.
	Set
	\[ U = \{ [\psi] \mid \psi(x)>0 \textrm{ and } \psi(y)>0 \} \subseteq S(G)\]
	This is clearly an open neighbourhood of $[\phi]$.	
	Suppose that $[\psi], [\psi'] \in U$.
	
	Let $A=A(g;\S,x)$, as in the proof of \cref{formula for BNS}. Since $\phi(y) \neq 0$, we can still form the matrix $B$ from \cref{formula for BNS}, and $[-\phi] \in \Sigma(G)$ \iff $A$ is invertible over
	$\hat{\Z G_{\phi}}$ \iff $B$ is invertible over
	$\hat{\Z G_{\phi}}$.
	
	Since $B$ is obtained from $A$ by an elementary row operation over $\mathfrak{F}(G, \phi)$ in which we add a multiple of the last row to another row, and such operations do not affect the canonical representative of the Dieudonn\'e determinant, we have
	\[ i_\phi(\det^c_{\D(G)}(A)) = \det^c_{\mathfrak{F}(G, \phi)}(A) = \det^c_{\mathfrak{F}(G, \phi)}(B) \]
	which is the product of the diagonal entries of $B$. Note that $B$ is invertible over $\hat{\Z G_\phi}$ if and only if the diagonal entries are invertible in $\hat{\Z G_\phi}$, which is the case if and only if their product is invertible in $\hat{\Z G_\phi}$ since $\hat{\Z G_\phi}$ is a domain. Thus, by \cref{mu}, $[-\phi]\in\Sigma(G)$ if and only if $\mu_\phi(\det^c_{\mathfrak{F}(G, \phi)}(B)) = \mu_\phi(i_\phi(\det^c_{\D(G)}(A))$ is of the form $\pm z$ for some $z\in G$.
	
	The same arguments apply to $\psi$ and $\psi'$ since $\psi(y) \neq 0 \neq \psi'(y)$. By \cref{mu}, it therefore suffices to prove
	\[ \mu_\psi(i_\psi(\det^c_{\D(G)}(A))) = \mu_{\psi'}(i_{\psi'}(\det^c_{\D(G)}(A)))\]
	
	If we put $d := \det^c_{\D(G)}(A)\cdot (x-1)^{-1}$, then this is equivalent to
	\[\mu_\psi(i_\phi(d)) = \mu_{\psi'}(i_\psi(d))\]
	since $\psi(x), \psi'(x) > 0$. But this is true by \cref{polytope and minimal} if we assume that $\psi$ and $\psi'$ are $d$-equivalent.
	
	\cref{invariants for fbc} (\ref{fbc:l2tor}) says that $d$ is mapped under the quotient maps
	\begin{equation*}
		\D(G)^\times \to \D(G)^\times / [\D(G)^\times,\D(G)^\times] \cong K_1^w(\Z G) \to \Wh^w(G)
	\end{equation*}
	to $-\tor_u(G)$, as desired. This finishes the proof in the case that $\ker \phi\neq F_2$.
	\smallskip
	
	Now suppose that $F_2 \leqslant \ker \phi$. Since $-\phi$ is not induced by the HNN extension, we must have $\phi(t) > 0$.
	
	Let us choose a generating set $x,y$ for $F_2$, and set
	\[ U = \big\{ [\psi] \mid \psi(t)> \vert \psi(z) \vert, z \in \supp \fox{g(y)}{y}  \big\} \]
	Again, this is an open neighbourhood of $[\phi]$.
	
	We proceed similarly to the previous case. Observing that $1-t$ is invertible over $\hat{\Z G_\psi}$ and $\hat{\Z G_{\psi'}}$ reduces the problem to verifying whether
	the matrix $A(g, \S, t)$ is invertible over $\hat{\Z G_\psi}$ and $\hat{\Z G_{\psi'}}$.
	The bottom-right entry of $A(g, \S, t)$ is $1 - t \fox{g(y)}y$, which is invertible for every $[\rho] \in U$ by construction. If $\psi$ and $\psi'$ are additionally $d$-equivalent for $d := \det^c_{\D(G)}(A(g, \S, t))\cdot (t-1)^{-1}$, we now continue in precisely the same way as before.
\end{proof}

\begin{rmk}
	Note that the result in the latter case also follows from the observation that $\Sigma(G)$ is open, since $[-\phi] \in \Sigma(G)$.
	
	Note also that our neighbourhood $U$ is very explicit, and rather large, especially when $\ker \phi \neq F_2$.
\end{rmk}

\begin{cor}
\label{finite bns}
Let $G = F_2*_g$ be a descending HNN extension. Then the Bieri--Neumann--Strebel invariant $\Sigma(G)$ has finitely many connected components.
\end{cor}
\begin{proof}
 Let $\phi \colon G \to \Z$ be the induced map. We know that $[-\phi] \in \Sigma(G)$, and so there exists an open set $U_\phi$ in $S(G)$ around $[-\phi]$ which lies entirely in $\Sigma(G)$. For all other non-trivial morphisms $\psi \colon G \to \R$ we obtain open sets $U_\psi$ as in the previous theorem. Since $S(G)$ is compact, we only need to look at finitely many open sets $U_{\psi_1}, \dots, U_{\psi_m}$. Thus it is enough to show that each such open set contains finitely many connected components of $\Sigma(G)$. This is clear for $U_\phi$, so let us assume that we are looking at $U_\psi$ with $[\psi] \neq [-\phi]$.
 
 The theorem above tells us that within $U_\psi$, lying inside of $\Sigma(G)$ is well-defined on the equivalence classes of the relation of being $d$-equivalent. Since there are only finitely many $d$-equivalence classes, the result follows.
\end{proof}

\section{UPG automorphisms}

In this section we will strengthen \cref{main thm alex vs thurston} and \cref{main thm BNS} for a class of free group automorphisms.

\begin{defin}[Polynomially growing and UPG automorphism]\label{def:upg} An automorphism $f:F_n\to F_n$ is \emph{polynomially growing} if the quantity $d(1,f^n(g))$ grows at most polynomially in $n$ for every $g\in F_n$, where $1$ denotes the identity in $G$ and $d$ is some word metric on $F_n$. If, additionally, the image $\bar{f}$ of $f$ under the obvious map $\Aut(F_n) \to\GL(n,\Z)$ is unipotent, i.e. $\id - \bar{f}$ is nilpotent, then $f$ will be called \emph{UPG}.
\end{defin}

The main result of Cashen-Levitt \cite[Theorem 1.1]{CashenLevitt2014} reads as follows.

\begin{thm}\label{thm CashenLevitt}
	Let $G = F_n\rtimes_g \Z$ with $n\geq 2$ and $g$ polynomially growing. Then there are elements $t_1, ..., t_{n-1}\in G\s- F_n$ such that
	\[ \Sigma(G) = -\Sigma(G) = \{[\phi]\in S(G)\mid \phi(t_i)\neq 0\text{ for all } 1\leq i\leq n-1\}\]
\end{thm}

Motivated by this, we prove

\begin{thm}\label{l2-torsion of upg}
	Let $G = F_n\rtimes_g \Z$ with $n\geq 1$ and $g$ a UPG automorphism. Denote by $p_k\colon G\to \Gamma_k = G/G_r^{k+1}$ the projection, where $G_r^{k}$ denotes the $k$-th subgroup of the rational derived series. For simplicity write $\Gamma_\infty$ for $G$ and $p_\infty$ for $\id_G$.
	
	Then there are elements $t_1, ..., t_{n-1}\in G\s- F_n$ which can be chosen to coincide with those of \cref{thm CashenLevitt} such that for $k\in\N\cup \{\infty\}$
	\begin{equation}\label{tor of upg} 
		\tor_u(G;p_k) = - \sum_{i=1}^{n-1} \; [\Z \Gamma_k\tolabel{p_k(1-t_i)} \Z \Gamma_k]
	\end{equation}
	In particular,
	\[ P_{L^2}(G;p_k) = \sum_{i=1}^{n-1} \; P(1-t_i) \in \Pol(H_1(G)_f)\]
	is a polytope (and not merely a difference of polytopes) which is independent of $k\in\N\cup \{\infty\}$.
\end{thm}

Combining the previous two results, we see that the BNS-invariant of UPG automorphisms is easily determined by their $L^2$-torsion polytope. More precisely, we have the following analogue of \cite[Theorem 1.1]{FriedlTillmann2015}.

\begin{cor}\label{main corollary}
	Let $G = F_n\rtimes_g \Z$ with $n\geq 2$ and $g$ a UPG automorphism. Let $\phi\in H^1(G;\R)$. Then $[\phi]\in \Sigma(G)$ if and only if $F_\phi(P_{L^2}(G))=0$ in $\Pol_T(H_1(G)_f)$.
\end{cor}

\begin{proof}
	Any one-dimensional face of
	\[P_{L^2}(G) = \sum_{i=1}^{n-1} \; P(1-t_i)\]
	contains a translate of $P(1-t_i)$ for some $1\leq i\leq n-1$.
	
	\smallskip
	Now $F_\phi(P_{L^2}(G))\neq 0$ \iff $F_\phi(P_{L^2}(G))$ contains a one-dimensional face, i.e. a translate of $P(1-t_i)$ for some $i$. This is equivalent to $\phi(t_i) = 0$ for some $i$, which by \cref{thm CashenLevitt} is equivalent to $[\phi]\notin \Sigma(G)$.
\end{proof}

\begin{rmk}\label{upg vs pol}
	We suspect \cref{l2-torsion of upg} to hold as well for polynomially growing automorphisms. It is well-known that any polynomially growing automorphism has a power that is UPG, see Bestvina--Feighn--Handel's \cite[Corollary 5.7.6]{Bestvinaetal2000}. Thus, in order to reduce \cref{l2-torsion of upg} for polynomially growing automorphisms to the case of UPG automorphisms, one needs a better understanding of the restriction homomorphism
	\[i^*\colon \Wh^w(F_n\rtimes_{g}\Z) \to \Wh^w(F_n\rtimes_{g^k}\Z)\]
	(induced by the obvious inclusion $i\colon F_n\rtimes_{g^k}\Z\to F_n\rtimes_{g}\Z\:$) since it maps $\tor_u(F_n\rtimes_g\Z)$ to $\tor_u(F_n\rtimes_{g^k}\Z)$ (see \cref{properties} (\ref{restriction})).
\end{rmk}

We also obtain

\begin{cor}\label{main corollary 2}
	Let $G = F_n\rtimes_g \Z$ with $n\geq 2$ and $g$ a UPG automorphism. Let $\phi\in H^1(G;\R)$. Then for all $k\in \N\cup \{\infty\}$ we have
	\[ \|\phi\|_A = \delta_k(\phi) = \|\phi\|_T.\]
\end{cor}
\begin{proof}
	This follows directly from the fact that $P_{L^2}(G;p_k)$ is independent of $k\in \N\cup\{\infty\}$ as stated in \cref{l2-torsion of upg}. Note that $b_1(G)\geq 2$ by \cite[Remark 5.6]{CashenLevitt2014}. Hence we get as special cases $P_{L^2}(G;p_0) = P_A(G)$ by \cref{invariants for fbc} (\ref{fbc:alex}) and this polytope determines the Alexander norm, and on the other hand $P_{L^2}(G;p_\infty) = P_{L^2}(G)$ which determines the Thurston norm.
\end{proof}

Theorems \ref{thm CashenLevitt} and \ref{l2-torsion of upg} both rely on the following lemma which follows from the train track theory of Bestvina--Feighn--Handel \cite{Bestvinaetal2000}; see \cite[Proposition 5.9]{CashenLevitt2014} for the argument.

\begin{lem}\label{splitting for upg}
	For $n\geq 2$ and a UPG automorphism $g\in \Aut(F_n)$, there exists $h\in\Aut(F_n)$ representing the same outer automorphism class as $g$, such that either
	\begin{enumerate}
		\item there is an $h$-invariant splitting $F_n = B_1\ast B_2, h= h_1\ast h_2$; or
		\item there is a splitting $F_n = B_1\ast \langle x\rangle$ such that $B_1$ is $h$-invariant and $h(x) = xu$ for some $u\in B_1$.
	\end{enumerate}
\end{lem}

This lemma allows us two write the semi-direct product associated to a UPG automorphism as an iterated splitting over infinite cyclic subgroups with prescribed vertex groups. This is explained in \cite[Lemma 5.10]{CashenLevitt2014} and will be repeated in the following proof.

\begin{proof}[Proof of \cref{l2-torsion of upg}]
	We prove the statement by induction on $n$. For the base case $n=1$ we have $F_1\rtimes_g \Z\cong \Z^2$ and $\tor_u(\Z^2;p_k) = 0$ for all $k\in\N\cup \{\infty\}$ by \cite[Example 2.7]{FriedlLueck2015b} which already verifies (\ref{tor of upg}). 

	For the inductive step, we may assume that $g=h$ in the notation of \cref{splitting for upg} since the isomorphism class of $F_n\rtimes_g\Z$ only depends on the outer automorphism class of $g$. We analyse the two cases appearing in \cref{splitting for upg} separately.
	\medskip
	
	\noindent \textbf{Case 1:} There is a $g$-invariant splitting $F_n = B_1\ast B_2, g= g_1\ast g_2$. Write
	\[G_i = B_i\rtimes_{g_i} \Z\]
	and let $G_0 = \Z\into G_i$ be the inclusion of the second factor. Then we have
	\[ G = F_n\rtimes_g\Z \cong G_1 *_{G_0}  G_2\]
	and the Fox matrix of $g$ is of the form
	\[ F(g) =
	\begin{pmatrix}
	F(g_1 )& 0  \\
	0 & F(g_2)
	\end{pmatrix} 
	\]
	Let $j_i\colon G_i\to G$ be the inclusions, and denote a generator of $G_0$ and its image in the various groups $G_i$ by $t$.
	
	By \cite[Remark 5.6]{CashenLevitt2014}, we have $b_1(G) \geq 2$ and similarly for $G_1$ and $G_2$. Hence by \cref{invariants for fbc} (\ref{fbc:higher alex}) and (\ref{fbc:alex}) as well as the above matrix decomposition, we compute in $\Wh^w(\Gamma_k)$
	\begin{align} \label{induction upg1}
	\begin{split}
	\tor_u(G;p_k) &= -[p_k(I-t\cdot F(g))] + [p_k(t-1)]\\
	&= -[p_k(I-t\cdot F(g_1))] - [p_k(I-t\cdot F(g_2))] + [p_k(t-1)]\\
	&= (j_1)_*(\tor_u(G_1;p_k^1))) + (j_2)_*(\tor_u(G_2;p_k^2)))  - [p_k(t-1)]\\
	\end{split}
	\end{align}
	where $p_k^i$ denote the projections on the quotients of the rational derived series of $G_i$. Here we have used that in our setting $p_k^i$ can be seen as a restriction of $p_k$.
	
	Denote the rank of $B_i$ by $r_i$. By the inductive hypothesis applied to $G_i$, there are elements
	\[t'_1,\dots, t'_{r_1-1} \in G_1 \s- B_1\]
	and
	\[t''_1,\dots , t''_{r_2-1} \in G_2 \s- B_2\]
	such that
	\begin{equation}\label{vertex group1}
	\tor_u(G_1;p_k^1) = -\sum_{i=1}^{r_1-1}\; [p_k^1(1-t'_i)]
	\end{equation}
	and 
	\begin{equation}\label{vertex group2}
	\tor_u(G_2;p_k^2) = -\sum_{i=1}^{r_2-1}\; [p_k^2(1-t''_i)]
	\end{equation} 
	Notice that $r_1 + r_2 = n$. Moreover, the corresponding induction step in the proof of \cref{thm CashenLevitt} adds $t$ to the union of the $t'_i$ and the $t''_i$. Thus the desired statement (\ref{tor of upg}) follows by combining (\ref{induction upg1}), (\ref{vertex group1}), and (\ref{vertex group2}).

	\medskip
	\noindent \textbf{Case 2:} There is a splitting $F_n = B_1\ast \langle x\rangle$ such that $B_1$ is $g$-invariant and $g(x) = xu$ for some $u\in B_1$. In this case, let $g_1 = g|_{G_1}$, $ G_1 = B_1\rtimes_{g_1}\Z \subseteq G$, and denote the stable letter of $G_1$ and $G$ by $t$. 

	In this case, the Fox matrix of $g$ takes the form
	\[ F(g) =
	\begin{pmatrix}
	F(g_1 )& 0  \\
	* & 1
	\end{pmatrix} 
	\]
	From this we compute in $\Wh^w(\Gamma_k)$ similarly as in the first case
	\begin{align} \label{induction upg}
	\begin{split}
		\tor_u(G;p_k) &= -[p_k(I-t\cdot F(g))] + [p_k(t-1)]\\
								&= -[p_k(I-t\cdot F(g_1))] - [p_k(1-t)] + [p_k(t-1)]\\
								&= \tor_u(G_1; p_k^1) - [p_k(1-t)]
	\end{split}
	\end{align}
	The corresponding induction step in the proof of \cref{thm CashenLevitt} adds $t$ to the elements $t'_i$ belonging to $G_1$ which we get from the induction hypothesis. 
	
	This finishes the proof of \cref{l2-torsion of upg}.
\end{proof}

\begin{rmk}
	The same strategy as above can be used to prove that the ordinary $L^2$-torsion $\tor (g) := \tor(F_n\rtimes_g \Z)\in\R$ vanishes for all polynomially growing automorphisms. Here the reduction to UPG automorphisms explained in \cref{upg vs pol} is simpler since we have $\tor(g^k) = k\cdot \tor(g)$, so that the vanishing of the $L^2$-torsion of some power of $g$ implies the vanishing of the  $L^2$-torsion of $g$. This is a special case of a result of Clay \cite[Theorem 5.1]{Clay2015}.
\end{rmk}

\bibliographystyle{math}
\bibliography{bibliography}

\bigskip

\noindent
\textsc{Florian Funke} \hfill \textsc{Dawid Kielak} \newline
Mathematisches Institut  \hfill Fakult\"at f\"ur Mathematik  \newline
Universit\"at Bonn  \hfill Universit\"at Bielefeld \newline
Endenicher Allee 60 \hfill Postfach 100131  \newline
D-53115 Bonn \hfill D-33501 Bielefeld \newline
Germany \hfill Germany \newline
\texttt{ffunke@math.uni-bonn.de} \hfill \texttt{dkielak@math.uni-bielefeld.de}

\end{document}